\documentclass[11pt]{amsart}
\usepackage[colorlinks=true, pdfstartview=FitV, linkcolor=blue, 
citecolor=blue]{hyperref}

\usepackage{amsmath,amssymb,bbm,amscd}
\usepackage{graphicx}
\usepackage{a4wide}

\theoremstyle{plain}
\newtheorem{theorem}{Theorem}[section]
\newtheorem{prop}[theorem]{Proposition}
\newtheorem{lemma}[theorem]{Lemma}
\newtheorem{coro}[theorem]{Corollary}
\newtheorem{fact}[theorem]{Fact}

\theoremstyle{definition}
\newtheorem{example}[theorem]{Example}
\newtheorem{remark}[theorem]{Remark}
\newtheorem{definition}[theorem]{Definition}

\numberwithin{equation}{section}

\newcommand{\ii}{\ts\mathrm{i}}
\newcommand{\ee}{\,\mathrm{e}}
\newcommand{\ts}{\hspace{0.5pt}}
\newcommand{\nts}{\hspace{-0.5pt}}

\DeclareMathOperator{\dens}{\mathrm{dens}}
\DeclareMathOperator{\card}{\mathrm{card}}
\DeclareMathOperator{\cent}{\mathrm{cent}}
\DeclareMathOperator{\supp}{\mathrm{supp}}
\DeclareMathOperator{\stab}{\mathrm{stab}}
\DeclareMathOperator{\norm}{\mathrm{norm}}
\DeclareMathOperator{\vol}{\mathrm{vol}}
\DeclareMathOperator{\GL}{\mathrm{GL}}
\DeclareMathOperator{\No}{\mathrm{N}}
\DeclareMathOperator{\imag}{\mathrm{Im}}

\newcommand{\vG}{\varGamma}

\newcommand{\vT}{\varTheta}

\newcommand{\fb}{\mathfrak{b}}

\newcommand{\cA}{\mathcal{A}}
\newcommand{\cB}{\mathcal{B}}

\newcommand{\cG}{\mathcal{G}}
\newcommand{\cH}{\mathcal{H}}
\newcommand{\cL}{\mathcal{L}}
\newcommand{\cP}{\mathcal{P}}
\newcommand{\cR}{\mathcal{R}}
\newcommand{\cS}{\mathcal{S}}

\newcommand{\cO}{\mathcal{O}}

\newcommand{\AAA}{\mathbb{A}}
\newcommand{\KK}{\mathbb{K}}
\newcommand{\ZZ}{\mathbb{Z}\ts}
\newcommand{\RR}{\mathbb{R}\ts}
\newcommand{\CC}{\mathbb{C}\ts}
\newcommand{\NN}{\mathbb{N}}
\newcommand{\QQ}{\mathbb{Q}}

\newcommand{\XX}{\mathbb{X}}
\newcommand{\YY}{\mathbb{Y}}

\newcommand{\one}{\mathbbm{1}}
\newcommand{\Aut}{\mathrm{Aut}}

\newcommand{\defeq}{\mathrel{\mathop:}=}

\newcommand{\exend}{\hfill $\Diamond$}

\newcommand{\myfrac}[2]{\frac{\raisebox{-2pt}{$#1$}}
      {\raisebox{0.5pt}{$#2$}}}

\begin{document}

\title[Shifts with small centraliser and large
  normaliser]{Number-theoretic positive entropy shifts\\[2mm]
  with small centraliser and large normaliser}

\author{M.~Baake}

\author{\'{A}.~Bustos}

\author{C.~Huck}

\address{Fakult\"{a}t f\"{u}r Mathematik,
  Universit\"{a}t Bielefeld,\newline \hspace*{\parindent}Postfach
  100131, 33501 Bielefeld, Germany}
\email{$\{$mbaake,abustos,huck$\}$@math.uni-bielefeld.de}

\author{M.~Lema\'{n}czyk}
\address{Faculty of Mathematics and Computer Science,
  Nicolaus Copernicus University,\newline
  \hspace*{\parindent}12/18 Chopin street,
87-100 Toru\'n, Poland}
\email{mlem@mat.umk.pl}

\author{A.~Nickel}

\address{Fakult\"{a}t f\"{u}r Mathematik,
  Universit\"{a}t Duisburg-Essen,
  \newline \hspace*{\parindent}Thea-Leymann-Str.~9,
  45127 Essen, Germany}
\email{andreas.nickel@uni-due.de}

\begin{abstract}
  Higher-dimensional binary shifts of number-theoretic origin with
  positive topological entropy are considered. We are particularly
  interested in analysing their symmetries and extended symmetries.
  They form groups, known as the topological centraliser and
  normaliser of the shift dynamical system, which are natural
  topological invariants. Here, our focus is on shift spaces with
  trivial centralisers, but large normalisers. In particular, we
  discuss several systems where the normaliser is an infinite
  extension of the centraliser, including the visible lattice points
  and the $k$-free integers in some real quadratic number fields.
\end{abstract}

\maketitle
\thispagestyle{empty}

\section{Introduction}\label{sec:intro}

Shift spaces under the action of $\ZZ^d$ form a much-studied class of
dynamical systems, both for $d=1$, compare \cite{LM}, and for
$d\geqslant 2$. In the latter case, much less is known in terms of
general classifications, and even subclasses such as those of
algebraic origin \cite{Klaus} are still rather enigmatic, despite
displaying fascinating facets that have been analysed intensely.  In
particular, one is looking for interesting topological invariants to
help analyse the jungle, and quite a bit of progress in this direction
has been made recently.

Among the available tools are the automorphism group of a shift space
and its various siblings and generalisations; see
\cite{CK,CQY,Don,DKKL,Baa,CP} and references therein. Here, we adopt
the point of view of \cite{BRY,Baa} to analyse both the (topological)
centraliser (denoted by $\cS$ below) \emph{and} the normaliser of the
shift space, the latter denoted by $\cR$, as this pair can be quite
revealing as soon as $d\geqslant 2$. In fact, both the topological
setting and the extension to higher dimensions go beyond some of the
initial studies \cite{Good,KLP} that specifically looked at
reversibility in the measure-theoretic setting for $d=1$; see
\cite{RQ,FS,Baa} and references therein for more on the early
reversibility results. Further, the groups $\cS$ and $\cR$ are often
explicitly accessible, both for systems of low complexity, where $\cS$
is often minimal due to some form of topological rigidity, and beyond,
where other rigidity mechanisms of a more algebraic nature emerge.

Below, we consider binary shift spaces of number-theoretic origin, as
motivated by recent progress on $\cB$-free systems and weak model
sets; see \cite{Abda,DKKL,KKL,BHS} and references therein.  By way of
characteristic examples with pure point spectrum, we demonstrate that
positive topological entropy may very well be compatible with small or
trivial centralisers, which means that $\cS$ agrees with the
underlying lattice (meaning a co-compact discrete subgroup of $\RR^d$)
or a finite{\ts}-index extension thereof, but also that such systems
may have considerably larger normalisers, which is particularly
interesting for $d\geqslant 2$. In fact, as shown in
\cite{BRY,Bustos}, it is the group $\cR$ that captures some obvious
symmetries, as visible from the chair tiling and related shift spaces
with their pertinent geometric symmetries.  Also, the computability 
of $\cS$ and $\cR$ in these cases can be an advantage over some 
of the more general, abstract (semi-)groups that are presently 
attracting renewed attention.  \smallskip

The paper is organised as follows. After the introduction of some
concepts and notions in Section~\ref{sec:prelim}, we set the scene
with the well-known example of the visible lattice points of $\ZZ^2$
in Section~\ref{sec:motiv}, leading to Proposition~\ref{prop:visible}
and Corollary~\ref{coro:vis-symm}, which in particular show that one
has $\cR = \ZZ^2 \rtimes \GL (2,\ZZ)$. Then, Section~\ref{sec:lattice}
states and proves this for $\ZZ^d$ with $d\geqslant 2$
(Theorem~\ref{thm:visible}) and introduces the general framework of
lattice{\ts}-based shift spaces, which can often be characterised by a
rather powerful admissibility condition for its elements
(Proposition~\ref{prop:locator}). Then, under some mild assumptions,
the normalisers are always maximal extensions of the corresponding
centralisers (Theorem~\ref{thm:max-extend}), with elements that are
affine mappings (Corollary~\ref{coro:affine}).

Section~\ref{sec:NT} explains the general number-theoretic setting of
an algebraic $\cB$-free system in higher dimensions, based on the
classic Minkowski embedding of (commutative) maximal orders and 
their ideals as lattices in $\RR^d$ for a suitable $d$. Here,
Theorem~\ref{thm:alg-symm} states the results on the triviality of
$\cS$ and the direct product nature of $\cR$, which are true under a
coprimality condition of the ideals chosen for $\cB$ and a mild
convergence condition, together known as the Erd\H{o}s property, in
generalisation of the one-dimensional notion \cite{DKKL} from
$\cB$-free integers.

Sections~\ref{sec:Gaussian} and \ref{sec:real} then cover some
paradigmatic examples from quadratic number fields.  In the complex
case, we treat the shift spaces generated by the $k$-free Gaussian or
the $k$-free Eisenstein integers (Theorems~\ref{thm:Gauss} and
\ref{thm:Eisenstein}). In both cases, $\cR$ is the extension of
$\cS\simeq \ZZ^2$ by a maximal finite subgroup of $\GL (2,\ZZ)$, which
is substantially different from the case of the visible lattice
points. Finally, in the real case, we consider $k$-free integers in
the maximal order of $\QQ (\sqrt{m}\,)$ for $m\in \{2,3,5\}$. Here,
Theorem~\ref{thm:real} states that $\cR$ is the semi-direct product of
$\cS\simeq \ZZ^2$ with a non-trivial \emph{infinite} subgroup of
$\GL (2,\ZZ)$, which can be given a clear interpretation in terms of
algebraic number theory. The latter case, which is the first example
of this type to the best of our knowledge, is intermediate between
known examples from inflation tilings and shifts such as that generated
by the visible lattice points. Thus, it looks particularly
promising for future work and extensions to general number fields.

\section{Preliminaries}\label{sec:prelim}

Let $\vG \subset \RR^{d}$ be a \emph{lattice} in $d$-space, that is, a
discrete and co-compact subgroup of $\RR^d$.  Below, we will be
working with the full shift (or configuration) space $\{0,1\}^{\vG}$,
equipped with the standard product topology, and certain of its closed
subspaces (called subshifts or simply shifts). We will generally use
$\XX$ to denote such a shift space, refering to either the full shift
or the subshift under consideration.  When the situation is
independent of the geometry of the lattice, we will choose
$\vG = \ZZ^{d}$ for simplicity. Any element $x\in\XX$ can also be
viewed as a subset of $\vG$, by taking the support of $x$, that is, by
mapping $x$ to
\[
  U^{}_{x} \, \defeq \, \supp (x) \, = \,
  \{ {n} \in \vG : x^{}_{{n}} = 1 \} \, \subseteq \, \vG .
\]
Conversely, any point set $U \subseteq \vG$ can be viewed as a
configuration, by mapping it to $x^{}_{U} = 1^{}_{U}$, that is, to its
characteristic function. As usual, $\XX$ admits a continuous action of
$\vG$ on it, defined by
$T \! : \, \vG \nts\nts \times \nts\nts \XX \xrightarrow{\quad} \XX$,
where $T (t, x) = T^{}_{t} \ts (x)$ with
$ \bigl( T^{}_{t}\ts (x)\bigr)_{{n}} \defeq x^{}_{{n} + t}$.  When
working with $\ZZ^d$, we shall usually refer to its standard basis as
$\{ {e}^{}_{1}, \ldots , {e}^{}_{d} \}$, and align this with the
elementary shift action of the $d$ commuting shift operators
$T^{}_{e_i}$. For the action of $\ZZ^d$ in this case, with
$t = (t^{}_{1}, \ldots , t^{}_{d})$, this simply means
$T^{}_{t} (x) = T^{\ts t_1}_{e_1} \cdots T^{\ts t_d}_{e_d} (x)$ for
all $x\in \{ 0,1 \}^{\ZZ^d}$.

Likewise, there is an action of $\vG$ on its subsets defined by
$\alpha^{}_{t} (U) = t + U \defeq \{ t + u : u \in U \}$.  It is easy
to check that $U_{T^{}_{t} (x)} = \alpha^{}_{-t} (U^{}_{x})$.  If we
denote the power set of $\vG$ by $\Omega$, we thus get the commutative
diagram
\begin{equation}\label{eq:cd-1}
\begin{CD}
  \XX @> T^{}_{t} >> \XX \\
  @V \gamma VV     @VV \gamma V \\
  \Omega @> \alpha^{}_{-t} >> \Omega
\end{CD}
\end{equation}
where $\gamma$ is the mapping defined by $x \mapsto U_{x}$. This is a
homeomorphism if we equip $\Omega$ with the \emph{local topology},
where two subsets of $\vG$ are $\varepsilon$-close to one another when
they agree on the ball of radius $1/\varepsilon$ around $0$.
Consequently, by slight abuse of notation, we will not distinguish
these two points of view whenever the context is clear. This means
that we will consider a subset $U \subseteq \vG$ simultaneously as a
configuration, and vice versa.

In this spirit, we can also consider the group of lattice automorphisms,
$\Aut (\vG) \simeq \GL (d, \ZZ)$. Indeed, if $\Aut (\XX)$ denotes the group 
of homeomorphisms of $\XX$, any $M\in \Aut (\vG)$ induces an element
\mbox{$h^{}_{\nts M} \in \Aut (\XX)$}, where
\begin{equation}\label{eq:hm-def}
  \bigl(h^{}_{\nts M} (x)\bigr)_{{n}}
  \, \defeq \, x^{}_{\nts M^{-1} {n}} \ts .
\end{equation}
In fact, the mapping $M \mapsto h^{}_{\nts M}$ defines an injective
group homomorphism. Here, one can check that
$U^{}_{h^{}_{\nts M} (x)} = M U^{}_{x}$, so the counterpart to
\eqref{eq:cd-1} is the commutative diagram
\begin{equation}\label{eq:cd-2}
\begin{CD}
  \XX @> h^{}_{\nts M} >> \XX \\
  @V \gamma VV     @VV \gamma V \\
  \Omega @> M  >> \Omega
\end{CD}
\end{equation}
which makes calculations with elements of the form $h^{}_{\nts M}$
more convenient in the formulation with subsets.  From now on, we
identify $\XX$ and $\Omega$, and use the symbol $\XX$ for both. To
ease the understanding, we will normally use $x$, $y$ for
configurations and $U$, $V$ for sets.

A \emph{point set} $S \subset \RR^d$, by which we mean an at most
countable union of singleton sets, is said to have \emph{natural
  density} if
\[
   \dens (S) \, = \lim_{r\to\infty}
   \frac{\card ( S\cap B_r )}{\vol (B_r)}
\]
exists, where $B_r $ denotes the closed ball of radius $r$ around
$0$. One can use other sets for averaging, as long as they are centred
around $0$ and satisfy some condition of F{\o}lner or van Hove type;
see \cite{TAO,BMP} for details.

Below, we shall need the following simple result on sublattices of a
given lattice, where the term \emph{sublattice} is meant to include
the property that the corresponding index is finite.

\begin{fact}\label{fact:cosets}
  Let\/ $\vG$ be a lattice in\/ $\RR^d$, and let\/ $\vG_1$ and $\vG_2$
  be sublattices of\/ $\vG$, with corresponding indices\/ $m^{}_{1}$
  and\/ $m^{}_{2}$, respectively. Then, $\vG_1 \cap \vG_2$ and\/
  $\vG_1 \nts + \nts \vG_2$ are sublattices of\/ $\vG$ as well.
   
  Further, if\/ the indices\/ $m^{}_{1}$ and $m^{}_{2}$ are coprime,
  one has\/ $\vG_1 + \vG_2 = \vG$, which implies that\/ $\vG_1$ meets
  all cosets of\/ $\vG_2$ and vice versa.
\end{fact}

\begin{proof}
If $[\vG : \vG_i] = m_i$, one has $m_i \vG \subseteq \vG_i$ by standard
arguments, which implies 
\[
   m^{}_{1} m^{}_{2} \vG \, \subseteq \, \vG_1 \cap \vG_2
   \, \subseteq \, \vG_1 + \vG_2 \, \subseteq \, \vG \ts .
\]
The sublattice property for $\vG_1 \cap \vG_2$ and
$\vG_1 + \vG_2$ is then clear.
   
The next statement is a consequence of what is sometimes
referred to as the diamond isomorphism theorem, but can also be
seen directly as follows. Set $n=[\vG : (\vG_1 \! + \! \vG_2) ]$ and 
$n_i = [ (\vG_1 \! + \!  \vG_2) : \vG_i]$. Then, for $i\in \{ 1,2\}$, 
\[
    m_i \, = \, [\vG : \vG_i ] \, = \,
    [\vG : (\vG_1 \! + \! \vG_2) ] \ts
    [ ( \vG_1  \! + \! \vG_2)  : \vG_i ]
    \, = \, n \ts\ts n_i \ts ,
\]
which implies $n | \gcd (m^{}_{1}, m^{}_{2}) = 1$ and thus
$\vG_1+\vG_2 = \vG$. The final implication for the cosets is a 
now simple consequence.
\end{proof}

An important concept for shift spaces is that of a block map; see
\cite{LM} for background.  When working with subshifts
$\XX \subseteq \cA^{\ts\ZZ^d}$ and $\YY \subseteq \cB^{\ts\ZZ^d}$ over
finite alphabets $\cA$ and $\cB$, a continuous mapping
$h \! : \, \XX \xrightarrow{\quad} \YY$ is called a \emph{block map}
if there is a non-negative integer $\ell$ such that, for every
$x\in\XX$ and all $n\in\ZZ^d$, the image $y=h(x)$ at position $n$ is
fully determined from the patch
$\big\{ x^{}_{n+m} : m \in [-\ell,\ell]^d \big\}$, that is, from the
knowledge of $x$ within a $d$-cube of sidelength $2\ell$ centred at
$n$. In other words, the action of $h$ can be seen as the result of a
sliding block code $\phi = \phi^{}_{h}$ that, for some fixed
$\ell\in\NN_0$, maps a cubic block of $(2 \ell+1)^d$ symbols from
$\cA$ to a single letter from $\cB$, positioned at the centre of the
block (which can easily be modified when needed). This is the symbolic
version of a \emph{local derivation rule} from discrete geometry
\mbox{\cite[Sec.~5.2]{TAO}}.  An important result that we shall need
repeatedly is the Curtis--Hedlund--Lyndon (CHL) theorem: If a
continuous mapping $h\! : \, \XX \xrightarrow{\quad} \YY$ intertwines
the shift action on $\XX$ and $\YY$, it must be a block map based on
some code $\phi$ of the above type \cite[Thm.~6.2.9]{LM}.

Below, we shall only be interested in subshifts on which the action of
$\vG$ is \emph{faithful}, which means that the subshift contains
non-periodic elements.  In this context, it is also natural to
consider the affine lattice group $\vG \rtimes \Aut (\vG)$, whose
elements $(t,M)$ act on $\RR^d$ via $(t, M) (y) \defeq M y + t$, and
correspondingly on $\XX$. In this formulation, the group
multiplication is $(t,M) (s,N) = (t+Ms, MN)$, with neutral element
$(0, \one)$ and inverse elements $(t, M)^{-1} = (-M^{-1} t ,M^{-1})$.
This group will become important later.

Further notation and concepts can now better be introduced along a
paradigmatic example, which will simultaneously motivate the various
extensions to follow.

\section{Visible lattice points and their shift
  space}\label{sec:motiv}

Consider the \emph{visible points} $V$ of the square lattice,
$\ZZ^{2}$, which are defined as
\[
      V \ts \defeq \, \{ (m,n) \in \ZZ^{2} : \gcd(m,n) = 1 \} \ts .
\]
They are also known as the primitive points, and are used in many
places; see also the cover page of \cite{Apo}.  Clearly, one has
$V\nts\nts =\ZZ^{2}\setminus \bigcup_{p\in \cP}(p\ts\ZZ^{2})$, where
$\cP$ denotes the set of rational primes.  Figure~\ref{fig:vis+gau}
shows a finite patch around the origin, in comparison with another set
that will be discussed later, in Section~\ref{sec:Gaussian}.  Let us
recall some well-known properties of $V$; see \cite{BMP,TAO} and
references therein for background and further results.

\begin{fact}
  The set\/ $\ts V\nts$ is uniformly discrete, but not relatively
  dense.  In particular, $V\nts$ contains holes of unbounded inradius
  that repeat lattice-periodically.  Yet, it satisfies\/
  $V\! - \nts V \nts\nts = \ZZ^{2}$ and has natural density\/
  $\ts\dens (V) = \frac{1}{\zeta(2)} = \frac{6}{\pi^{2}}$, where\/
  $\zeta (s)$ is Riemann's zeta function.
    
  Furthermore, the set\/ $V\nts$ is pure point diffractive, with the
  diffraction measure being invariant under the action of the affine
  group\/ $\ZZ^{2}\nts \rtimes \GL (2,\ZZ)$.  \qed
  \end{fact}
  
\begin{figure}
\begin{center}
   \includegraphics[width=0.8\textwidth]{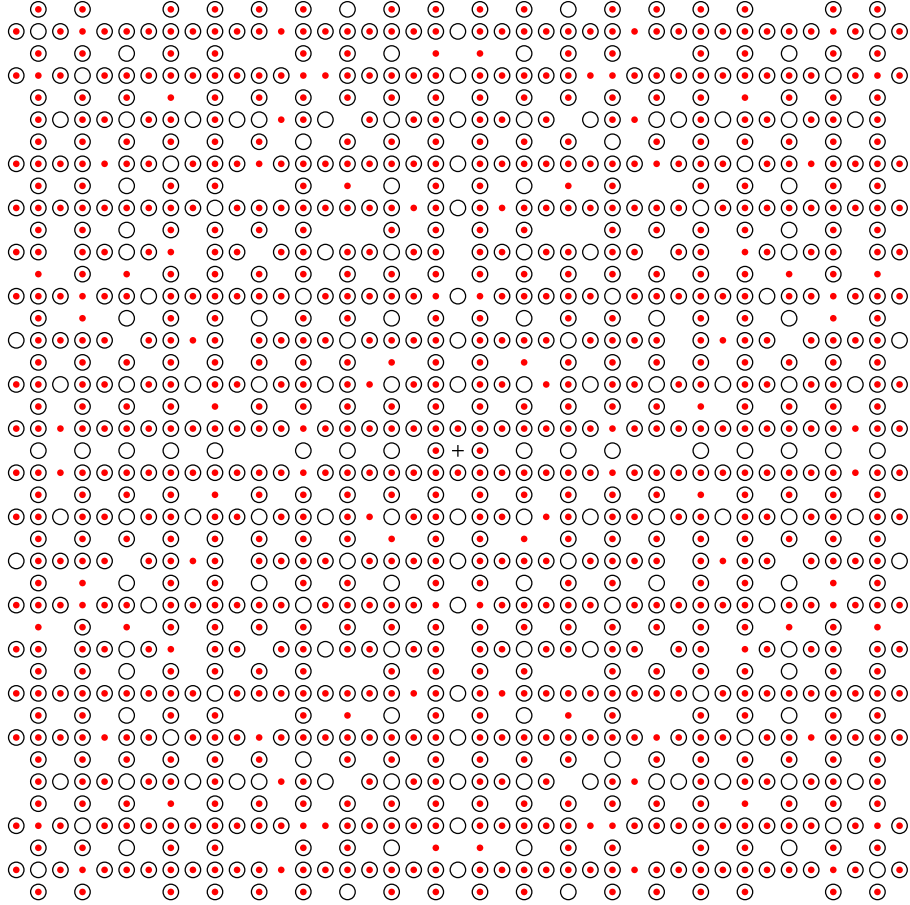}
\end{center}
\caption{\label{fig:vis+gau}Central patch of the visible points of
  $\ZZ^{2}$ (dots) and of the square{\ts}- free Gaussian integers
  (circles). The cross in the centre marks the origin. }
\end{figure}

Now, let $\XX^{}_{V} \defeq \overline{\ZZ^{2} + V}$ be the orbit
closure of $V$ under the shift (or translation) action of $\ZZ^{2}$,
where the closure is taken in the standard product topology, also
known as the \emph{local topology} due to its geometric
interpretation: Two configurations (or subsets) are close if they
agree on a large neighbourhood of $0\in\ZZ^2$. In particular, since
$V$ has holes of arbitrary size, one immediately obtains that
$\varnothing\in\XX^{}_{V}$, where $\varnothing$ is the empty set and
represents the all-$0$ configuration. Clearly, $\XX^{}_{V}$ is a
compact space, which is canonically identified with a subshift in
$\{ 0,1 \}^{\ZZ^{2}}$, and $\bigl( \XX^{}_{V},\ZZ^{2} \bigr)$ is a
topological dynamical system.

Call a subset of $\ZZ^{2}$ \emph{admissible} if it misses at least one
coset modulo $p\ts\ZZ^{2}$ for any $p \in \cP$.  One easily verifies
that the set of admissible sets constitutes a subshift of
$\{ 0,1 \}^{\ZZ^{2}}$ as well, denoted by $\AAA$.  Since the set $V$
by the remark above misses the zero coset modulo each $p\ts\ZZ^{2}$,
one readily verifies that the elements of $\XX^{}_{V}$ are admissible,
so $\XX^{}_{V} \subseteq \AAA$. In fact, it was shown in
\cite[Lemma~4]{BH} that $V$ shows all cosets except the zero coset
modulo each $p\ts\ZZ^{2}$, and is thus a maximal element of
$\XX^{}_{V}$. Further, one has the following result, the first part of
which will be generalised below on the basis of
Propositions~\ref{prop:locator} and \ref{prop:locator-2}.

\begin{prop}[\cite{BH}\label{prop:visible}]
  The space\/ $\XX^{}_{V}$ coincides with the shift space of
  admissible sets, $\AAA$. In particular, $\XX^{}_{V}$ is hereditary\/
  $(\nts$closed under the formation of subsets$\ts\ts )$. The
  topological dynamical system $\bigl( \XX^{}_{V},\ZZ^{2} \bigr)$ has
  topological entropy $\frac{6}{\pi^{2}}\nts\log (2)$.
  
  With respect to the existing natural frequency measure\/
  $\nu^{}_{\mathrm{M}}$, which is also known as the Mirsky measure,
  the measure-theoretic dynamical system\/
  $\bigl( \XX^{}_{V}, \ZZ^{2}, \nu^{}_{\mathrm{M}} \bigr)$ has pure
  point dynamical spectrum, but trivial topological point spectrum.
    
  The measure\/ $\nu^{}_{\mathrm{M}}$ is ergodic for the\/
  $\ZZ^{2}$-action, and\/ $V$ is a generic element for\/
  $\nu^{}_{\mathrm{M}}$ in\/ $\XX^{}_{V}$. Moreover, the
  measure-theoretic entropy for\/ $\nu^{}_{\mathrm{M}}$ vanishes. \qed
\end{prop}

The characterisation of a number-theoretic shift space via an
admissibility condition was originally observed by Sarnak for the
square{\ts}-free integers, and later extended to Erd\H{o}s $\cB$-free
numbers in \cite{Abda} and generalised to the lattice setting in
\cite{PH}. Since this step is vital to us, we later present a
streamlined version of the proof that covers the generality we need.

\begin{remark}
  The generating shifts induce unitary operators on the Hilbert space
  $L^2 (\XX^{}_{V}, \nu^{}_{\mathrm{M}} )$, and the simultaneous
  eigenfunctions form a basis of this space \cite{BHS,BL}.  Except for
  the trivial one, no other eigenfunction is continuous. However, as
  follows from a recent result by Keller \cite{Keller}, see also the
  discussion in \cite{BHS,BSS}, there is a subset of $\XX^{}_{V}$ of
  full measure on which the eigenfunctions \emph{are} continuous.
  This is related to the fact that $V$ is a weak model set of maximal
  density \cite{BHS} in the cut and project scheme $(\RR^2, H, \cL)$
  with compact internal group $H = \prod_{p\in\cP} \ZZ^2/p\ts\ZZ^2$
  and the lattice $\cL$ being the diagonal embedding of $\ZZ^2$ into
  $\RR^2 \nts\nts \times \nts\nts H$.  It is an interesting open
  problem to understand the missing null set, and to connect it with
  the rather intricate relation between the topological and the
  measure{\ts}-theoretic structure of this dynamical system.  \exend
\end{remark}

Let $\Aut (\XX^{}_{V})$ be the automorphism group of $\XX^{}_{V}$, by
which we mean the group of all homeomorphisms of $\XX^{}_{V}$,
irrespective of whether they commute with the generators $T_1, T_2$ of
the $\ZZ^{2}$-action or not.  The translation action of $\ZZ^{2}$ on
$\XX^{}_{V}$ is faithful, wherefore we have
$\cG \defeq \langle T_{1}, T_{2} \rangle \simeq \ZZ^{2}$. Clearly,
$\XX^{}_{V}$ is not the full shift, and $\varnothing$ is the only
fixed point of $\XX^{}_{V}$ under the translation action, since
$\ZZ^2$ (as the all-$1$ configuration) is not an element of
$\XX^{}_{V}$.

The \emph{symmetry group} of $\XX^{}_{V}$, see \cite{Baa} and
references therein for background, is
\[
  \cS (\XX^{}_{V}) \, \defeq \, \cent^{}_{\Aut(\XX^{}_{V})} (\cG) \, =
  \, \{ H \in \Aut (\XX^{}_{V}) : G H = H G \text{ for all } G \in \cG
  \} \ts ,
\]
which clearly contains $\cG$ as a normal subgroup. This centraliser is
often called the automorphism group of the subshift, denoted as
$\Aut (\XX^{}_{V}, \cG)$, but we prefer to avoid the potential
confusion with the automorphism group in the above (or Smale)
sense. For any $S\in \cS (\XX^{}_{V})$, $S(V)$ has a dense shift orbit
in $\XX^{}_{V}$ (as $V$ has dense orbit by definition). Moreover, one
has $S(\varnothing)=\varnothing$ since $S(\varnothing)$ can also be
seen as a fixed point under the translation action.

Now, consider an arbitrary $S\in \cS (\XX^{}_{V})$.  By the CHL
theorem, there is a block code (or map)
$\phi \!:\,\{0,1\}^{[-\ell,\ell]^2}\xrightarrow{\quad} \{0,1\}$ of a
suitable size (parameterised by $\ell$) such that, for any
$x\in \XX^{}_{V}$, the value of $Sx$ at a position $k\in\ZZ^{2}$ is
given by the value under $\phi$ of the corresponding block of $x$
around this very position, which we call its \emph{centre}. This means
\[
    (S x)^{}_{k} \, = \,
    \phi\big(x^{}_{[k+[-\ell,\ell]^2]}\big) ,
\] 
where $x^{}_{[k+[-\ell,\ell]^2]}(m)=x^{}_{k+m}$ for
$m\in [-\ell,\ell]^2$. Since $S(\varnothing)=\varnothing$, it is clear
that $\phi \bigl( 0^{}_{[-\ell,\ell]^2} \bigr) = 0$.

Next, following \cite{BRY,Baa}, we define the \emph{extended symmetry
  group} as
\[
     \cR (\XX^{}_{V}) \, \defeq \, \norm^{}_{\Aut(\XX^{}_{V})} (\cG) \, = \,
     \{ H \in \Aut (\XX^{}_{V}) :  H \cG H^{-1} = \ts \cG \ts \} \ts ,
\]
which contains both $\cG$ and $\cS (\XX^{}_{V})$ as normal subgroups.
Every $H\in \cR (\XX^{}_{V})$ must satisfy
$H(\varnothing) = \varnothing$, as $H(\varnothing)$ can once again be
shown to be fixed under any element of $\cG$. Since every (extended)
symmetry induces an automorphism of $\cG \simeq \ZZ^2$ via the
conjugation action, $\cR (\XX^{}_{V})$ can at most be a group
extension of $\cS (\XX^{}_{V})$ by $\Aut (\ZZ^2 ) = \GL (2,\ZZ)$.

Let us state the final result for this specific example, which is
a special case of our more general statement 
(Theorem~\ref{thm:visible}) in the next section.

\begin{coro}\label{coro:vis-symm}
  For the topological dynamical system\/
  $\bigl( \XX^{}_{V},\ZZ^{2} \bigr)$, the symmetry group is the
  minimal one, so\/ $\cS (\XX^{}_{V}) = \cG \simeq \ZZ^{2}$, while the
  extended symmetry group is
\[
      \cR (\XX^{}_{V}) \, = \, \cS (\XX^{}_{V}) \rtimes \Aut (\ZZ^{2}) 
      \, \simeq \, \ZZ^{2} \rtimes \GL (2,\ZZ) \ts ,
\]     
hence the maximal extension possible.  \qed
\end{coro}

This result shows that positive (topological) entropy is very well
compatible with a minimal centraliser, while the factor group
$\cR (\XX^{}_{V})/\cS (\XX^{}_{V})$ need neither be a finite nor a
periodic group, where the latter statement implies that the factor
group contains elements of infinite order. This combination can also
occur for subshifts with zero entropy, as can be seen from the
subshift that is obtained as the orbit closure of a singleton
configuration and contains the shift orbit of this configuration
together with the all-$0$ configuration; compare \cite[Ex.~4.3]{TAO}.

\section{General lattice setting}\label{sec:lattice}

The statement of Corollary~\ref{coro:vis-symm} 
is not restricted to $d=2$. Indeed, one has the following 
generalisation; see \cite{BMP,BH,PH} for its first part.

\begin{theorem}\label{thm:visible}
  Let\/
  $V = \{ (n^{}_{1}, \ldots , n^{}_{d}) \in \ZZ^{d} : \gcd (n^{}_{1},
  \ldots , n^{}_{d}) =1\}=\ZZ^{d}\setminus \bigcup_{p}(p\ts\ZZ^{d})$ be
  the set of visible points of\/ $\ZZ^{d}$, with\/ $d\geqslant 2$, and
  consider the topological dynamical system\/
  $\bigl( \XX^{}_{V}, \ZZ^{d} \bigr)$ with\/
  $\XX^{}_{V} = \overline{\ZZ^{d} + V}$.  Then, $\XX^{}_{V}$ has
  topological entropy\/ $\log(2)/\zeta(d)$ and satisfies\/
  $\XX^{}_{V} = \AAA$, where\/ $\AAA$ consists of all admissible
  subsets of\/ $\ZZ^d$, that is, all subsets\/ $U\nts \subset \ZZ^d$
  such that, for every\/ $p\in\cP$, $U$ misses at least one coset
  modulo\/ $p \ts\ZZ^d$.
  
  The symmetry group, or topological centraliser, of this system is\/
  $\cS (\XX^{}_{V}) = \ZZ^{d}$, while its extended symmetry group, or
  topological normaliser, is\/
  $\cR (\XX^{}_{V}) = \ZZ^{d}\nts \rtimes \GL (d,\ZZ)$.  
\end{theorem}

\begin{proof}
  The statement on the centraliser is a rigidity result that is driven
  by the identity $\XX^{}_{V}=\AAA$, which also forces $\XX^{}_{V}$
  to be hereditary.  It follows from a slight modification
  of the argument put forward in \cite{Mentzen}, which we repeat here
  in a form that is tailored to the higher-dimensional lattice systems
  we consider here and below. It employs a lattice version of the
  Chinese remainder theorem (CRT) based on the pairwise coprime
  sublattices of the form $p\ts\ZZ^{d}$ ($p$ prime) of the integer
  lattice. Note that the solutions of a system of congruences appear
  lattice{\ts}-periodically, which guarantees some flexibility
  regarding the actual position of solutions in the square
  lattice. This argument also works for general lattices. \smallskip

  We start from the identity $\XX^{}_{V} = \AAA$, which follows from
  Proposition~\ref{prop:visible} together with its generalisation in
  Proposition~\ref{prop:locator} and Theorem~\ref{thm:max-extend}
  below. First, we show that any symmetry $S\in \cS (\XX^{}_{V})$ acts
  on the singleton set $U^{}_{0}=\{0\} \in \XX^{}_{V}$ as a
  translation, that is, $S(U^{}_{0})=U^{}_{0}+k$ for some $k\in\ZZ^d$,
  where $U^{}_{0} \in \XX^{}_{V}$ follows from $\XX^{}_{V} = \AAA$.
  Since $S$ is a homeomorphism that commutes with the shift action, it
  corresponds to a block code $\phi$, by the CHL theorem.  Here and in
  what follows, we identify any subset of $\ZZ^d$ with its
  characteristic function, and thus with a binary configuration, as
  explained in Section~\ref{sec:prelim}. Then,
  $S (U^{}_{0}) = U^{}_{0} + k$ is equivalent to saying that $\phi$
  takes the value $1$ on exactly one block with singleton support. For
  the latter, note first that $\phi$ cannot take the value $0$ on all
  blocks with singleton support, as this would imply
  $S(U^{}_{0})=\varnothing$ which is impossible ($S$ is invertible and
  we already have $S(\varnothing)=\varnothing$).
  
  Assuming the existence of two different blocks with singleton
  support that are sent to $1$ by the code, there is a prime $p$ and
  an admissible set $U\nts\subset \ZZ^d$ of cardinality $p^{\ts d} -1$ that
  comprises all cosets modulo $p\ts\ZZ^d$ except the zero coset, together
  with the property that $S(U)$ shows \emph{all} possible cosets
  modulo this very $p\ts\ZZ^{d}$ and is thus no longer admissible.  To
  see this, $p$ is chosen such that the difference $n$ of the centres
  of the two blocks (a non-zero element of $\ZZ^d$) does \emph{not}
  belong to $p\ts\ZZ^{d}$. In fact, by the CRT, the $p^{\ts d}-1$ elements of
  $U$ can be chosen arbitrarily well separated from one another.
  Then, the assertion follows because $S(U)$ will contain a translate
  of $U \cup (n + U)$ and, since $S(U)$ is admissible, a translate of
  this union is contained in $V\!$. Consequently, for some $m\in\ZZ^d$,
  both $U+m$ and $(U+n)+m$ consist of $p^{\ts d}-1$ elements and 
  are equal modulo $p\ts\ZZ^{d}$ (both showing all non-zero cosets modulo 
  $p\ts\ZZ^{d}\ts $) --- a contradiction to $n\neq 0$ modulo $p\ts\ZZ^{d}$
  from the construction.
  
  After replacing $S$ by $S^{\ts\prime} \defeq T^{}_{k} \circ S$, 
  so that $S^{\ts\prime} (U^{}_{0}) = U^{}_{0}$, and slightly enlarging 
  the size of the block code, one can assume that the only block with 
  singleton support that is sent to $1$ is the block that has value $1$ 
  only at $0$. One is then left to show that $S^{\ts\prime}=\mathrm{id}$. 
  For convenience, we now rename $S^{\ts\prime}$ by $S$, and
  show that $S=\mathrm{id}$.

  This follows from the maximality of $V$ together with the crucial
  observation that $S(U)\subseteq U$ (equivalently
  $U \subseteq S^{-1}(U)$) for all $U \in \XX^{}_{V}$, due to the
  properties of the block code for $S$ just established. 
  So, any (automatically admissible)
  block of $1^{}_{U}$ with value $0$ at its central position is sent
  to $0$ by the code. This claim can be shown by an argument similar
  to the one used above.  Assume the existence of an admissible block
  $C$ with value $0$ at its centre that is sent to $1$ by the
  code. This block then appears in $V$ at a position $s$ with
  $s\in p\ts\ZZ^d$ for a suitable $p$. Again, one can choose a set $U$ of
  $p^{\ts d}-1$ elements of $V$ that shows all cosets except the zero coset
  modulo $p\ts\ZZ^{d}$. By the CRT, we may assume that these $p^{\ts d}-1$
  elements are well separated and also well separated from $s$
  (together with the whole block $s+C$ of $V$ at $s$). It is then
  immediate that $U\cup (s+C)$ is admissible and that $S(U\cup (s+C))$
  will contain the set $U\cup\{s\}$ and thus shows all cosets modulo
  $p\ts\ZZ^{d}$, a contradiction. \smallskip

  It remains to determine the normaliser.  Since $\cG \simeq \ZZ^{d}$,
  with $\Aut (\ZZ^d) = \GL (d,\ZZ)$, there is a group homomorphism
\[
   \psi \! : \, \cR (\XX^{}_{V}) \xrightarrow{\quad} \GL (d,\ZZ)
\]
that is induced as follows. If $H \in \cR (\XX^{}_{V})$, we have
$H \cG H^{-1} = \cG$, so a set of generators of $\cG$ must be mapped
to a (possibly different) set of generators under the conjugation
action. Starting from our canonical choice,
$\cG = \langle T^{}_{e_1}, \ldots, T^{}_{e_d} \rangle$, one finds
$H T^{}_{i} H^{-1} = \prod_{j} T^{m^{}_{ji}}_{j}$ where the
$m^{}_{ji}$ are the matrix elements of $M^{}_{\nts H} = \psi (H)$.  It
is routine to verify the homomorphism property. In particular, with
$T^{}_{n} = T^{n^{}_1}_{e_1} \cdots T^{n^{}_d}_{e_d}$, one gets
\begin{equation}\label{eq:H-action}
    H T^{}_{n} H^{-1}  \, = \, T^{}_{M^{}_{\nts H} \ts n} \ts .
\end{equation}    

For $M\in\GL (d,\ZZ)$, in line with Eq.~\eqref{eq:hm-def}, define the
mapping $H^{}_{\! M}$ on $\XX=\{0,1\}^{\ZZ^d}$ by
\[
      ( H^{}_{\! M} x )^{}_{{n}} \, = \, x^{}_{M^{-1} {n}} \ts ,
\]
which clearly is a homeomorphism of $\XX$. Now, each $M$ maps our 
set $V$ onto itself, as $\GL (d,\ZZ)$ acts transitively on $V$. 
Consequently, also the orbit $\{ t+V : t\in \ZZ^d \}$ is mapped
onto itself by $M$, hence $M$ preserves $\XX_{V}$ by continuity.  In
other words, invoking \eqref{eq:cd-2}, we see that $H^{}_{\! M}$ is an
element of $\cR (\XX^{}_{V})$, and that
\[
     1 \, \xrightarrow{\quad} \, \cS (\XX^{}_{V})
     \, \xrightarrow{\; \mathrm{id} \;} \, \cR (\XX^{}_{V})
     \, \xrightarrow{\,\, \psi \,\,} \, \GL (d,\ZZ)
     \, \xrightarrow{\quad} \, 1
\]
is a short exact sequence. Moreover, the mapping
\[
  \varphi \! : \, \GL (d,\ZZ) \xrightarrow{\quad} \Aut (\XX^{}_{V} )
\]
defined by $\varphi (M) = H^{}_{\! M}$ is a group homomorphism as
well, with $\psi (H^{}_{\! M}) = M$. Consequently,
$\cH \defeq \varphi (\GL (d,\ZZ))$ is a subgroup of $\cR (\XX^{}_{V})$
that is isomorphic with $\GL (d,\ZZ)$. Since $\varphi\circ \psi$ acts
as the identity on $\cH$, our claim follows.
\end{proof}

\begin{remark}\label{rem:gen-vis}
  With respect to the patch frequency (or Mirsky) measure, the
  situation is also the same as for $d=2$, meaning that the dynamical
  spectrum of $(\XX^{}_{V}, \ZZ^{d}, \nu^{}_{\mathrm{M}})$ is pure
  point, with trivial topological point spectrum. Nevertheless, the
  measure{\ts}-theoretic eigenfunctions are continuous on a subset of
  $\XX^{}_{V}$ of full measure; see the discussion in
  \cite{BHS}. \exend
\end{remark}

In fact, the above multi-dimensional setting allows for a further
generalisation.

\begin{definition}\label{def:B-free-lattice}
  Let\/ $\cB = \{ b_{i} \mid i \in \NN \}$ be an infinite set of
  positive integers that is \emph{primitive} in the sense that\/
  $b_{i} | b_{j}$ implies\/ $i=j$. Consider the point set\/
  $V^{}_{\cB} = \ZZ^{d} \setminus \bigcup_{i\in\NN} b_{i} \ts \ZZ^{d}$
  in\/ $\RR^{d}$, and define\/
  $\XX^{}_{\cB} = \overline{\ZZ^{d} + V^{}_{\cB}}$, which is compact.
  Then, the dynamical system\/ $(\XX^{}_{\cB}, \ZZ^{d})$ is called a\/
  $\cB$-\emph{free lattice system}. It is called \emph{Erd\H{o}s} when
  the $b_{i}$ are pairwise coprime and satisfy
\[
    \sum_{i=0}^{\infty}\myfrac{1}{b_i^d} \, < \, \infty\ts ,
\]
which is an additional condition only for $d=1$.
\end{definition}

Note that $d=1$ is the case of $\cB$-free systems in $\ZZ$, which is
extensively studied in \cite{DKKL,KKL} and references therein. The
primitivity condition really is some irreducibility notion, as any
multiple of some $b_{i}$ could simply be removed from the set $\cB$
without any effect on $V^{}_{\cB}$.  It is obvious that $\ZZ^{d}$ in
Definition~\ref{def:B-free-lattice} can be replaced by any lattice
$\vG \subset \RR^{d}$. However, since this does not change the
arithmetic situation at hand, we restrict our attention to $\ZZ^{d}$
for now.

A set $U \subset\ZZ^d$ is called \emph{admissible} for $\cB$ if, for
every $b\in\cB$, $U$ meets at most $b^{\ts d} - 1$ cosets of the
sublattice $b\ts \ZZ^d$. Equivalently, $U$ is admissible if it misses
at least one coset of $b \ts \ZZ^d$ for each $b\in\cB$.  The set of
all admissible subsets of $\ZZ^d$ is again denoted by $\AAA$, and
constitutes a subshift.  By definition, $V^{}_{\cB} \in \AAA$, and we
thus have $\XX^{}_{\cB} \subseteq \AAA$.  If $P$ and $Q$ are disjoint
finite subsets of $\ZZ^d$, we define the \emph{locator set}
\[
     L (P,Q) \, \defeq \, \{ t \in \ZZ^d : t+P \subset V^{}_{\cB}
     \text{ and } t+Q \subset \ZZ^d \setminus V^{}_{\cB} \}
\]
in analogy to the treatment in \cite{PH}. One has the following
connection, which is a generalisation of both \cite[Prop.~2.5]{Abda}
and \cite[Thm.~2]{PH}.

\begin{prop}\label{prop:locator}
  Assume that\/ $\bigl( \XX^{}_{\cB}, \ZZ^d \bigr)$ is Erd\H{o}s, and
  let\/ $P$ and\/ $Q$ be disjoint finite subsets of\/ $\ZZ^d$. Then,
  the following properties are equivalent.
\begin{enumerate}\itemsep=2pt
\item  $L(P,Q)$ has positive natural density.
\item  $L(P,Q)\ne \varnothing$.
\item  $P$ is admissible for\/ $\cB$. 
\end{enumerate}
\end{prop}

\begin{proof}
  The implication $(1) \Rightarrow (2)$ is clear.  If
  $L(P,Q)\ne \varnothing$, one has $t+P \subset V^{}_{\cB}$ for some
  $t\in\ZZ^d$, so $t+P \in \AAA$ and hence $P\in\AAA$, which shows
  $(2) \Rightarrow (3)$.
  
  It remains to prove $(3) \Rightarrow (1)$. To this end, let $m =
  \card (P)$ and set
\[
    S^{}_{1} \, \defeq \, \{ b \in \cB : \min \bigl(
    \card ( P \bmod b ),   b^{\ts d} -1 \bigr)  < m \} \ts ,
\]  
which is a finite subset of $\cB$. Further, for the elements $q\in Q$,
select distinct elements $b_q$ from $\cB \setminus S^{}_{1}$, and set
$S^{}_{2} = \{ b_q : q \in Q \}$. Without loss of generality, we may
choose each $b^{}_{q}$ large enough so that $p\equiv q \bmod b^{}_{q}$
has no solution with $p\in P$, which is to say that $q$ is a
representative of a coset modulo $b^{}_{q}$ that is missed by
$P$. Since $\card (S^{}_{2}) = \card (Q)$,
$S \defeq S^{}_{1} \cup S^{}_{2}$ is still a finite subset of $\cB$,
with $S=S^{}_{1}$ for $Q=\varnothing$.

Since $P$ is admissible for $\cB$, we know that, for each $b\in\cB$,
at least one coset of $b\ts \ZZ^d$ is missed by $P$. Let $p^{}_b$ be
a representative of this coset, where we may choose $p^{}_{b} = q$ for
all $b=b^{}_{q}\in S^{}_{2}$ due to our choice of $S^{}_{2}$.  As our
system is Erd\H{o}s, we can invoke the lattice version of the CRT to
see that there is an element $t^{}_0 \in \ZZ^d$ such that
\[
     t^{}_0  \equiv - p^{}_b \bmod b \ts ,
     \quad \text{for all $b\in S$} \ts .
\]     
Note that, with the choice of the $p^{}_{b}$ for $b\in S^{}_{2}$ just
made, this comprises the congruences
$ t^{}_0 \equiv - q \bmod b^{}_q $ for all $q\in Q$.  In fact, due to
the pairwise coprimality, we know that the set of \emph{all} solutions
is given by the lattice coset $t^{}_{0} + c \ts\ZZ^d$ with
$c = \prod_{b\in S} b$.  For any $t$ from this coset and then every
$b\in S$, we thus have $t+p \not\equiv 0 \bmod b$, which is to say
that $t+P$ avoids the zero coset for all $b\in S$, while
$t+q \equiv 0 \bmod b^{}_{q}$, so no element of $t+Q$ can lie in 
$V^{}_{\cB}$.

Now, let $R_n \defeq \{ b \in \cB \setminus S : b \leqslant n \}$,
which is finite, where we assume the integer $n$ to be large enough so
that $R_n \ne \varnothing$. Now, consider
\[
    \vT_n \, \defeq \, \bigl( t^{}_{0} + c \ts \ZZ^d \bigr)
    \cap \{ t \in \ZZ^d : t \not \equiv - p \bmod b
    \text{ for all } b \in R_n \text{ and all } p \in P \} \ts .
\]
The second set is a finite union of translates of the lattice
$\gamma^{}_{n} \ZZ^d$ with $\gamma^{}_{n} = \prod_{b\in R_n} b$.
Invoking Fact~\ref{fact:cosets}, it is clear that $\vT_n$ consists of
finitely many cosets of the intersection lattice, which is
$c \ts\ts \gamma^{}_{n}\ts \ZZ^d$, and thus has a well-defined natural
density. Consequently, $\vT_n$ has density
\[
  \dens (\vT_n) \, = \, c^{-d} \prod_{b\in R_n}
  \Bigl( 1 - \frac{\card (P)}{b^{\ts d}} \Bigr)
\]
because, modulo $b$ for any $b\in R_n$, no two points of
$P$ can be equal by our choice of $S^{}_{1}$. 

Each term in the product is a positive number, again due to our choice
of $S^{}_{1}\subseteq S$, so the Erd\H{o}s condition guarantees that
the infinite product satisfies
\[
    \prod_{b\in \cB\setminus S}
    \Bigl( 1 - \frac{\card (P)}{b^{\ts d}} \Bigr)
    \, = \, D \, > \, 0 \ts ,
\]
which is to say that it converges to a positive number.  Since
$\vT^{}_{n+1} \subseteq \vT^{}_{n}$ for all large enough $n$, say
$n \geqslant n^{}_{0}$, we can take the limit $n\to\infty$ and
conclude that
$\vT_{\infty} \defeq \bigcap_{n\geqslant n^{}_{0}} \vT_n$ is a set of
solutions of our congruence conditions, for all $b \in \cB$, with
positive natural density. So, for any $t\in \vT_{\infty}$, we have
$t+P \subset V^{}_{\cB}$ together with $t+Q \subset \ZZ^d 
\setminus \! V^{}_{\cB}$ as claimed.
\end{proof}

\begin{theorem}\label{thm:max-extend}
  Let\/ $\bigl( \XX^{}_{\cB}, \ZZ^{d} \bigr)$ be a\/ $\cB$-free
  lattice system, with symmetry group\/ $\cS = \cS
  (\XX^{}_{\cB})$. Then, the group of extended symmetries is given
  by\/ $\cR = \cR (\XX^{}_{\cB}) = \cS \rtimes \GL (d, \ZZ)$, which is
  to say that the extension is always the maximally possible one.

  If\/ $\bigl( \XX^{}_{\cB}, \ZZ^{d} \bigr)$ is Erd\H{o}s, one has\/
  $\XX^{}_{\cB} = \AAA$, the system is hereditary, and it has minimal
  symmetry group, $\cS = \cG \simeq \ZZ^{d}$, and we thus get\/
  $\cR = \ZZ^{d} \rtimes \GL (d,\ZZ)$.
\end{theorem}

\begin{proof}
  Due to the assumptions, any $\cB$-free lattice system defines a
  shift, with faithful shift action, wherefore its symmetry group,
  $\cS (\XX^{}_{\cB})$, contains a normal subgroup that is isomorphic
  with $\ZZ^{d}$, namely the one generated by the shift action itself,
  $\cG$.

  Since $\Aut(\ZZ^{d}) = \GL (d,\ZZ)$, any $M\in\GL (d,\ZZ)$ maps
  $\ZZ^{d}$ onto itself, hence one also has
  $M(b\ts \ZZ^{d}) = b\ts M(\ZZ^{d}) = b\ts \ZZ^{d}$ for any
  $b\in\cB$. This implies $M (V^{}_{\cB}) = V^{}_{\cB}$. We thus see
  that $\cH \defeq \varphi ( \GL (d,\ZZ))$ is a subgroup of
  $\cR (\XX^{}_{\cB})$ that is isomorphic with $\GL (d,\ZZ)$.  Since
  we have $\psi (\cH) = \psi (\cR(\XX^{}_{\cB})) = \GL (d,\ZZ)$, where
  $\psi$ is the group homomorphism from above, $\cS$ is the kernel of
  the group endomorphism $\varphi\circ \psi$.  By construction,
  $\varphi \circ \psi$ acts as the identity on $\cH$, and the claimed
  semi-direct product structure follows.
  
  Clearly, we have $\XX^{}_{\cB} \subseteq \AAA$, as explained
  earlier.  For the converse inclusion, when $\XX^{}_{\cB}$ is
  Erd\H{o}s, consider an arbitrary $S\in\AAA$ and, for $n\in\NN$, set
  $S_n = S \cap B_n (0)$, which is finite.  By
  Proposition~\ref{prop:locator}, for each $n\in\NN$, there exists
  some
  $t_n \in L \bigl(S_n, (\ZZ^d \cap B_n (0) )\setminus S_n \bigr) \ne
  \varnothing$, which means that
  \[
      (V^{}_{\cB} - t^{}_n ) \cap B_n (0)
      \, = \, S^{}_n \ts .
  \]
  Consequently, $\lim_{n\to\infty} (V^{}_{\cB} - t^{}_n) = S$ in the
  local topology, and $S\in\XX^{}_{\cB}$. This shows
  $\AAA\subseteq \XX^{}_{\cB}$ and hence $\XX^{}_{\cB} =
  \AAA$. Clearly, $\XX^{}_{\cB}$ is then also hereditary. Now, a
  straight-forward modification of the centraliser argument used in
  the proof of Theorem~\ref{thm:visible} establishes $\cS = \ZZ^d$.
\end{proof}

Alternatively, the structure of the last proof can be summarised in
stating that
\[
    1 \, \xrightarrow{\quad} \, \cS \,
    \xrightarrow{\; \mathrm{id} \; } \,
    \cR \, \xrightarrow{\,\, \psi \,\, } \, \GL (d,\ZZ) \,
    \xrightarrow{\quad} \, 1
\]
is a short exact sequence where $\cH \defeq \varphi(\GL(d,\ZZ))$ is a
subgroup of $\cR$ with $\cH \simeq \GL (d,\ZZ)$ and the property that
$\varphi \circ \psi$ acts as the identity on $\cH$.  Outside the class
of Erd\H{o}s $\cB$-free lattice systems, the centraliser can indeed be
a finite-index extension of $\cG$, as is known from one-dimensional
examples of Toeplitz type \cite{Keller-pc}, but we do not consider
this case below.

\begin{example}
  Let $k\in\NN$ be fixed and consider the lattice $\ZZ^d$. Then,
  $\cB = \{ p^k : p\in\cP \}$ leads to the $k$-free lattice points in
  $d$ dimensions, which is Erd\H{o}s for $k \ts d \geqslant 2$. They
  have been studied from various angles in \cite{BMP,PH,BH}, and
  provide a natural extension of our motivating example from
  Section~\ref{sec:motiv}.
  
  In particular, one always obtains a measure{\ts}-theoretic dynamical
  system $\bigl( \XX^{}_{V_{\cB}}, \ZZ^d, \nu^{}_{\mathrm{M}} \bigr)$
  with pure point diffraction and dynamical spectrum, as in
  Remark~\ref{rem:gen-vis}. The topological entropy is
  $\log (2) / \zeta (k d)$, while the measure{\ts}-theoretic entropy
  with respect to the natural patch frequency (or Mirsky) measure
  $\nu^{}_{\mathrm{M}}$ always vanishes \cite{PH}, as it must in view
  of the fact that the dynamical spectrum of
  $\bigl( \XX^{}_{V_{\cB}}, \ZZ^d, \nu^{}_{\mathrm{M}} \bigr)$ is pure
  point. \exend
\end{example}

The result of Theorem~\ref{thm:max-extend} can more generally be
looked at as follows. Let $\bigl( \XX,\ZZ^d \bigr)$ be a faithful
shift, with centraliser $\cS (\XX)$ and normaliser $\cR (\XX)$, and
assume that $h^{}_{\nts M} \in \Aut (\XX)$ for some
$M\in \GL (d,\ZZ)$, where $h^{}_{\nts M}$ is the mapping defined in
Eq.~\eqref{eq:hm-def}.  Let $T^{}_{n}$ with $n\in\ZZ^d$ denote the
shift by $n$ as before, so
$\bigl( T^{}_{n} x \bigr)_{m} = x^{}_{m+n}$, and consider an element
$H\in \cR (\XX)$ with $M=\psi (H)$. Then, for any $\ell\in\ZZ^d$, one
obtains the commutative diagram
\begin{equation}\label{eq:cd-3}
  \begin{CD}
  \XX @> H >> \XX @> \, h_{\nts M^{-1}} \, >> \XX \\
  @V T^{}_{\ell} VV     @V T^{}_{\nts M\ell} VV @VV T^{}_{\ell} V \\
  \XX @> H >> \XX @> \, h_{\nts M^{-1}} \, >> \XX
\end{CD}
\end{equation}
from Eq.~\eqref{eq:H-action}, where $h^{}_{\nts M^{-1}}\in \Aut (\XX)$
by assumption. In particular,
$h^{}_{\nts M^{-1}}\circ H \in \Aut (\XX)$ commutes with the shift
action, hence is a block map by the CHL theorem.

At this point, the structure of the centraliser enters crucially, and
one obtains an interesting consequence as follows, where
$\psi \! : \, \cR (\XX) \xrightarrow{\quad} \Aut (\ZZ^d)$ is the
homomorphism from above.

\begin{coro}\label{coro:affine}
  Let\/ $\bigl( \XX,\ZZ^d \bigr)$ be a faithful subshift with trivial
  centraliser.  Consider an element\/ $H\in\cR (\XX)$ with\/
  $h^{}_{\psi(H)} \in \Aut(\XX)$. Then, $H$ is an affine mapping and\/
  $h^{}_{\psi(H)}\in\cR (\XX)$.
\end{coro}

\begin{proof}
  From the diagram \eqref{eq:cd-3}, with $M=\psi (H)$, we know that
  $h^{}_{\nts M^{-1}} \circ H \in \cS (\XX)$, so this mapping equals
  $T^{}_{n}$ for some $n\in\ZZ^d$.  This means
  $H = h^{}_{\nts M} \circ T^{}_{n}$, which acts as
\[
    (Hx)^{}_{m} \, = \, x^{}_{M^{-1} m + n} \ts .
\]
The equivalent formulation with sets, due to the relation
$h^{}_{\nts M} \circ T^{}_{n} = T^{}_{\nts Mn} \circ h^{}_{\nts M}$,
now reads $H(U) = -Mn + M(U)$, which is affine.

Finally, since $H\in\cR (\XX)$, one also has
$h^{}_{\nts M} = H \nts \circ T^{}_{-n} \in \cR (\XX)$.
\end{proof}

The occurrence of affine mappings in the context of $\ZZ^d$-actions,
as a sign of some degree of rigidity, is also known from
\cite[Thm.~1.1]{KS}, and will become important later.

\section{Number theoretic setting}\label{sec:NT}

The concept of a $\cB$-free lattice system from
Definition~\ref{def:B-free-lattice} is only one possibility to
generalise the one{\ts}-dimensional notion. For another, combining
methods from the theory of aperiodic order \cite{TAO} with classic
results from elementary and algebraic number theory \cite{BS,N}, one
may start with the treatment of square{\ts}-free integers in algebraic
number fields as in \cite{CV}, and simplify and generalise it as
follows.

Let $\KK$ be an algebraic number field of degree $d$, so
$[\KK : \QQ] = d < \infty$. Let $\cO$ be the ring of integers in
$\KK$, which is the unique maximal order in $\KK$, such as $\ZZ$ for
$\KK=\QQ$, $\ZZ[\ii]$ for $\KK = \QQ(\ii)$, or $\ZZ[\sqrt{2}\,]$ for
$\KK = \QQ(\sqrt{2}\,)$. Let
$\iota\!:\, \cO\xrightarrow{\quad} \RR^{r} \! \times \nts \nts
\CC^{s}$ be the mapping defined by
\[
    z\mapsto \bigl(\rho^{}_{1}(z),\dots,\rho^{}_{r}(z),
    \sigma^{}_{1}(z),\dots,\sigma^{}_{\nts s}(z) \bigr) ,
\]
where
$\rho^{}_{1},\dots,\rho^{}_{r}$ are the real embeddings of $\KK$ into
$\CC$, while $\sigma^{}_{1},\dots,\sigma^{}_{\nts s}$ arise from the
complex embeddings 
of $\KK$ into $\CC$ by choosing exactly one embedding
from each pair of complex conjugate ones (in particular, we have
$d=r + 2s$). Clearly, depending on $\KK$, one either takes
$\rho^{}_{1}$ or $\sigma^{}_{1}$ to be the identity.

Now, if $\fb$ is a non-zero ideal of $\cO$, its absolute \emph{norm}
is defined by $\No (\fb) \defeq [ \cO : \fb]$. In fact, for any of the
above choices, the image $\iota(\fb)$ is a lattice in
$\RR^{r} \times \CC^{s} \simeq \RR^d$, and the absolute norm of $\fb$
is precisely the index of the sublattice $\iota(\fb)$ in the lattice
$\iota(\cO)$, and thus a finite number. The map $\iota$ is usually
called the \emph{Minkowski embedding}\/ of $\cO$\/; see
\cite{BS,N,TAO} for details.

To continue, let $\KK$ be an algebraic number field of degree $d$, and
$\cO$ its ring of integers, with Minkowski embedding
$\vG= \iota (\cO) \subset\RR^{d}$.  Let
$\cB = \{ \fb_{i} \mid i \in \NN \}$ be an infinite set of non-trivial
ideals of $\cO$, where $\cB$ is assumed to be \emph{primitive} in the
sense that $\fb_{i} \supseteq \fb_{j}$ implies\/ $i=j$. Let
$\vG_{i} =\iota (\fb_i)$ and consider
$V^{}_{\cB} \defeq \vG \setminus \bigcup_{i\in\NN}
\vG_{i}\subset\RR^{d}$, which thus is the Minkowski embedding of
$\cO \setminus \bigcup_{i\in\NN} \fb_{i}$, and define its hull as the
orbit closure
\[
     \XX^{}_{\cB} \, = \, \overline{\vG + V^{}_{\cB}} 
\]
in the local topology, so $\XX^{}_{\cB}$ is compact as in our 
previous examples.
   
\begin{definition}   
  In the setting just explained, the topological dynamical system
  $\bigl( \XX^{}_{\cB}, \vG \bigr)$ is called an \emph{algebraic}
  $\cB$-\emph{free lattice system}, or simply an \emph{algebraic}
  $\cB$-\emph{free system}.
   
  Such a system is called \emph{Erd\H{o}s} when the $\fb_{i}$ are
  pairwise coprime (meaning $\fb_{i}+\fb_{j}=\cO$ for all $i\neq j$)
  and satisfy
\[
    \sum_{i=1}^{\infty}\myfrac{1}{\No (\fb_{i})} \, < \, \infty\ts .
\] 
\end{definition}

As before, we call a set $U \subset \vG$ \emph{admissible} for $\cB$
when, for every $\fb\in\cB$, the set $U$ meets at most $\No (\fb) - 1$
cosets of $\vG^{}_{\fb} \defeq \iota (\fb)$ in $\vG$, that is, misses
at least one. All admissible subsets of $\vG$ once again constitute a
subshift, denoted by $\AAA$, which contains $\XX^{}_{\cB}$ by
construction.

\begin{prop}\label{prop:locator-2}
  Assume that\/ $\bigl( \XX^{}_{\cB}, \vG \bigr)$ is Erd\H{o}s, and
  let\/ $P$ and\/ $Q$ be disjoint finite subsets of\/ $\vG$. Then, the
  following properties are equivalent.
\begin{enumerate}\itemsep=2pt   
\item The locator set\/  $L(P,Q) \defeq
   \{ t \in \vG : t+P \subset V^{}_{\cB} \text{ and } t + Q \subset
   \vG \setminus V^{}_{\cB}\}$ has positive natural density.
\item $L (P,Q)   \ne \varnothing$.
\item $P$ is admissible for\/ $\cB$.
\end{enumerate}
\end{prop}

\begin{proof}
  This is a variant of the proof of Proposition~\ref{prop:locator},
  where $(1) \Rightarrow (2) \Rightarrow (3)$ is again clear. We thus
  need to establish $(3) \Rightarrow (1)$.

  Let $m = \card (P)$ and choose $S^{}_{1}$ as the set of all ideals
  $\fb \in \cB$ such that $\card (P \bmod \vG^{}_{\fb}) < m$ or
  $\No (\fb) \leqslant m$. As before, $S^{}_{1}$ is finite.  Then, for
  $S^{}_{2}$, select distinct ideals from $\cB \setminus S^{}_{1}$,
  denoted by $S^{}_{2} = \{ \fb^{}_{q} : q \in Q \}$, where, without
  loss of generality, we may select ideals $\fb^{}_{q}$ of
  sufficiently large absolute norm such that $P$ does not meet the
  coset modulo $\fb^{}_{q}$ represented by $q$. Then, consider
  $S = S^{}_{1} \cup S^{}_{2}$, which is still finite. As all ideals
  $\fb\in\cB$ can be viewed as lattices $\vG^{}_{\fb}$ via the
  Minkowski embedding, we can again invoke the CRT to find an element
  $t^{}_{0} \in \vG$ so that
\[
     t^{}_{0}  \equiv - p^{}_{\fb} \bmod \vG^{}_{\fb} \ts ,
     \quad  \text{for all $\fb \in S$} \ts ,
\]
where $p^{}_{\fb}$ is a representative of a coset modulo
$\vG^{}_{\fb}$ that is missing in $P$, which we know to exist.  Due to
our construction of $S^{}_{2}$, we may choose $p^{}_{\fb^{}_{q}} = q$
for all $q\in Q$, wherefore the above congruences actually comprise
$t^{}_{0} \equiv - q \bmod \vG^{}_{\fb^{}_{q}}$ for all $q\in Q$.  By
pairwise coprimality of the $\fb\in\cB$, we see that the set of all
solutions is the coset $t^{}_{0} + G$, where $G$ is the Minkowski
embedding of the ideal $\prod_{\fb\in S} \fb$.  For any $t$ from this
coset, $t+P$ avoids the zero coset of $\vG^{}_{\fb}$ for all
$\fb\in S$, while no element of $t+Q$ is in $V^{}_{\cB}$, so
$t+Q \subset \vG \setminus V^{}_{\cB}$.

Now, for $n\in\NN$, consider the set
$R_{n} \defeq \{ \fb \in \cB\setminus S :\No (\fb) \leqslant n \}$.
For a suitable $n^{}_{0}$ and then all $n\geqslant n^{}_{0}$, the set
$R_n$ is non-empty and finite. Next, define
\[
   \vT_n \, \defeq \, ( t^{}_{0} + G )
  \cap \{ t \not\equiv - p \bmod \vG^{}_{\fb} \text{ for all }
  \fb \in R_n \text{ and all } p \in P \} \ts .
\]
Then, $\vT_n$ is once again a finite union of translates of a
non-trivial intersection lattice and thus a set of positive natural
density, the latter being given by
\[
   \dens ( G )  \prod_{\fb \in R_n} \Bigl( 1 -
   \frac{\card (P)}{\No (\fb)} \Bigr) .
\]
As in the previous case, the product is convergent as $n\to\infty$ by
the Erd\H{o}s condition, so
$\vT_{\infty} \defeq \bigcap_{n\geqslant n^{}_{0}} \vT_n$ is a subset
of $\vG$ of positive density such that, for any $t\in \vT_{\infty}$,
we have $t+P \subset V^{}_{\cB}$ and
$t+Q \subset \vG \setminus V^{}_{\cB}$.
\end{proof}

\begin{theorem}\label{thm:alg-symm}
  An Erd\H{o}s algebraic\/ $\cB$-free system\/
  $\bigl( \XX^{}_{\cB}, \vG \bigr)$ satisfies\/ $\XX^{}_{\cB} = \AAA$
  and is hereditary. Moreover, it has minimal symmetry group, which
  means\/ $\cS = \cS (\XX^{}_{\cB}) = \vG \simeq \ZZ^d$. Moreover, its
  extended symmetry group is of the form\/
  $\cR (\XX^{}_{\cB}) = \cS \rtimes \cH$, where\/ $\cH$ is isomorphic
  to a non-trivial subgroup of\/ $\GL (d,\ZZ)$.
\end{theorem}

\begin{proof}
  While $\XX^{}_{\cB} \subseteq \AAA$ is clear,
  $\AAA\subseteq\XX^{}_{\cB}$ is shown exactly as in
  Theorem~\ref{thm:max-extend}, this time on the basis of
  Proposition~\ref{prop:locator-2}, so $\XX^{}_{\cB} = \AAA$, and this
  shift is hereditary.  Then, the statement on the centraliser
  follows, once again, from a straight-forward modification of the
  argument used in the proof of Theorem~\ref{thm:visible}.

  Let $\XX = \XX^{}_{\cB}$, and consider an arbitrary $H\in\cR (\XX)$.
  Here, we have $M\defeq \psi (H) \in \Aut (\vG)$ in analogy to our
  previous cases, and the diagram \eqref{eq:cd-3} changes to
\begin{equation}\label{eq:cd-4}
 \begin{CD}
  \XX @> H >> \XX @> \, h_{\nts M^{-1}} \, >> \YY \\
  @V T^{}_{\ell} VV     @V T^{}_{\nts M\ell} VV @VV T^{}_{\ell} V \\
  \XX @> H >> \XX @> \, h_{\nts M^{-1}} \, >> \YY
\end{CD}
\end{equation}
where $\YY \defeq h^{}_{\nts M^{-1}} (\XX)$, while $T^{}_{\ell}$ with
$\ell\in\vG$ is the shift in this case. Note that both $\XX$ and $\YY$
are subshifts of $\{ 0,1 \}^{\vG}$, on which $T^{}_{n}$ and
$h^{}_{\nts M}$ are still well defined, and it is clear that
$\varnothing \in \XX \cap \YY$. This new diagram is again commutative,
so $\chi = h^{}_{\nts M^{-1}} \circ H$ intertwines the shift actions
on $\XX$ and $\YY$. Consequently, by the CHL theorem, $\chi$ is a
block map.

The space $\YY$ inherits important properties from $\XX$, such as its
characterisation through admissibility (now defined via the images of
cosets in $\XX$ under $h^{}_{\nts M^{-1}}$) as well as being
heredi\-tary.  After minor modifications, the arguments from the proof
of Theorem~\ref{thm:visible} now show that $\chi$ must be a shift map,
hence equal to $T^{}_{n}$ for some $n\in\vG$.  But
$T^{}_{n} \ts \XX = \XX$, whence $H\in \Aut (\XX)$ now implies
\[
    \YY \, = \, h^{}_{\nts M^{-1}} (\XX) \, = \,
    h^{}_{\nts M^{-1}} ( H \XX ) \, = \, T^{}_{n} \ts \XX
    \, = \, \XX \ts ,
\]
and we are back to the situation of
Corollary~\ref{coro:affine}. Consequently, $H$ is an affine mapping,
with $H = h^{}_{\nts M} \circ T^{}_{n}$, and
$h^{}_{\nts M} \in \cR (\XX)$.  We thus have a short exact sequence
\[
    1 \, \xrightarrow{\quad} \, \cS (\XX)
    \, \xrightarrow{\;\, \text{id} \; \, } \, \cR (\XX)
    \, \xrightarrow{\;\, \psi \; \, } \,
    \cH \defeq \psi \bigl( \cR (\XX) \bigr)
    \, \xrightarrow{\quad} \, 1
\] 
with $\cH$ a subgroup of $\Aut (\vG)$. In particular, we get
$\cR (\XX) = \cS (\XX) \rtimes \cH$ as claimed.

To see that $\cH$ is non-trivial, we observe that the unit group
$\cO^{\times}$ is non-trivial (it contains at least the elements
$\pm 1$) and, via the Minkowski embedding, isomorphic to a subgroup of
$\Aut (\vG) \simeq \GL (d,\ZZ)$.  Each element of $\cO^{\times}$ maps
any ideal $\fb$ onto itself, so the corresponding mapping induced by
the Minkowski embedding is a bijection of $V^{}_{\cB}$, and thus gives
rise to an extended symmetry. Further elements emerge from non-trivial
Galois automorphisms of $\KK$, such as complex conjugation when $\KK$
is a totally complex extension of $\QQ$. Consequently, the claim on
the nature of $\cH$ is clear.
\end{proof}

\begin{remark}
  The systems covered by Theorem~\ref{thm:alg-symm} show many
  similarities with the $k$-free lattice points discussed earlier. In
  particular, they have positive topological entropy, which can in
  principle be determined from their description as weak model sets of
  maximal density in the sense of \cite{BHS}. The spectral properties
  will reflect the comments made in Remark~\ref{rem:gen-vis}.  We
  leave details to the interested reader.  \exend
\end{remark}

Unlike the situation in Theorem~\ref{thm:max-extend}, the group $\cH$
will generally \emph{not} be $\Aut (\vG) \simeq \GL (d, \ZZ)$, as we
shall see in Section~\ref{sec:Gaussian} below. In particular, for
$M\in \Aut(\vG)$ and $\fb\in\cB$, it need not be true that
$M(\fb) = \fb$ or $M (V^{}_{\cB}) = V^{}_{\cB}$. The following
negative result, obtained via methods from analytic number theory, was
pointed out to us by Valentin Blomer~\cite{Blomer}.

\begin{fact}\label{fact:Blomer}
  Let\/ $M \in \GL (2, \ZZ)\setminus \mathrm O(2, \ZZ)$. Then, there
  exist Gaussian primes\/ $\rho \in \ZZ[\ii] \simeq \ZZ^2$ such that a
  positive proportion of square-free Gaussian integers\/
  $\alpha \in \ZZ[\ii]$ satisfies\/ $\rho^2\mid M\alpha$, and\/
  $M\alpha$ is thus not square-free in\/ $\ZZ[\ii]$. \qed
\end{fact}

As we shall see in the next section, a simpler statement of purely
algebraic nature exists, which suffices for our purposes and permits
various generalisations.

\section{Power-free Gaussian and Eisenstein
  integers}\label{sec:Gaussian}

From now on, we shall need some classic results on quadratic number
fields, which can all be drawn from \cite[Chs.~14 and 15]{HW} or from
\cite{Zagier}. To keep things simple, we only consider rings of
integers that are Euclidean, so that we can easily work with primes
and prime factorisation (up to units) rather than with ideals; see
\cite{BBN} for various generalisations in our context.

As an example of an algebraic $\cB$-free system that is Erd\H{o}s, let
us view $\ZZ^{2}$ as $\ZZ[\ii]$, the ring of Gaussian integers, and
consider, for some fixed $2\leqslant k \in \NN$, the subset of
$k$-free elements (to be defined below).  $\ZZ[\ii]$ is the maximal
order in the quadratic field $\QQ (\ii)$, and is Euclidean.  The
\emph{unit group} of $\ZZ[\ii]$ is
\[
  \ZZ[\ii]^{\times} \, = \, \{ 1, \ii, -1, -\ii \}
  \, \simeq \, C^{}_{4} \ts .
\]
If $\cP$ denotes the set of rational primes as before, the Gaussian
primes \cite[Thm.~252]{HW} can be represented by
\[
  \cP^{}_{\mathrm{G}} \, = \, \{ 1\! + \! \ii \}
  \cup \{ p \in \cP : p \equiv 3 \bmod 4 \}
  \cup \{ \pi, \bar{\pi} : \pi \bar{\pi} = p \in
  \cP \text{ with } p \equiv 1 \bmod 4 \} \ts ,
\]
where $\bar{\cdot}$ is complex conjugation.  The three subsets
correspond to the ramified prime, where $(1+\ii)^2 = 2 \ts \ii$, the
inert primes, and the (complex) splitting primes, respectively.
Within the last, by slight abuse of notation, we assume one
representing pair for each $p$ to be selected, for instance by
demanding $\pi$ to lie in the positive quadrant. This way, the
representation of the primes is unique, and prime factorisation works
up to units.

Now, for any integer $k\geqslant 2$, we can define
$V^{(k)}_{\mathrm{G}}$ as the set of Gaussian integers that are not
divisible by the $k\ts$th power of any Gaussian prime. This is the set
of $k$-free Gaussian integers.  Figure~\ref{fig:vis+gau} contains an
illustration of the set $V^{(2)}_{\mathrm{G}}\! $, which was also used
in \cite{CV}. We begin with a geometric symmetry consideration of
$V^{(k)}_{\mathrm{G}}$ as follows.

\begin{lemma}\label{lem:gauss-group}
  Let\/ $k\geqslant 2$ be fixed and let\/
  $A \! : \, \ZZ[\ii] \xrightarrow{\quad} \ZZ[\ii]$ be a\/
  $\ZZ$-linear bijection that maps\/
  $V \nts\nts = V^{(k)}_{\mathrm{G}} \!$ into itself,
  $A(V) \subseteq V\!$.  Then, $A$ is a bijection of\/
  $\ts U \nts\nts = \ZZ[\ii]^{\times}$, and of\/ $V\!$ as well. As
  such, it is of the form\/ $A (x) = \epsilon \ts\ts \sigma (x)$
  with\/ $\epsilon \in U$ and\/
  $\sigma \in \{ \mathrm{id}, \overline{\ts\cdot\ts} \}$, that is,
  $\epsilon$ is a unit and\/ $\sigma$ a field automorphism of\/
  $\QQ (\ii)$.

  Together, these mappings form a group, which is the stabiliser of\/
  $V\!$ in\/ $\GL (2,\ZZ)$, denoted by\/ $\ts\stab (V)$. The latter,
  for any\/ $k\geqslant 2$, is the dihedral group\/
  $D_4 \simeq C_4 \rtimes\ts C_2$ of order\/ $8$, which is the
  symmetry group of the square and as such a maximal finite subgroup
  of\/ $\GL(2,\ZZ)$.
\end{lemma}

\begin{proof}
  Clearly, any $A$ of the form $A(x) = \epsilon \ts\ts \sigma (x)$
  maps units to units, and $V$ onto itself.  Conversely, if $A$
  preserves $U$ and $A(1)= \epsilon$, bijectivity of $A$ implies
  $A(\ii) = \ii\ts \epsilon$ or $A(\ii) = -\ii\ts \epsilon$, and
  $\ZZ\ts$-linearity of $A$ determines the image of any
  $x\in \ZZ[\ii]$ from here. In the first case, this gives
  $A(x) = \epsilon \ts x$, and $A(x) = \epsilon \ts \bar{x}$ in the
  second.  It thus remains to show that any $\ZZ$-linear bijection $A$
  of $\ZZ[\ii]$ with $A(V)\subseteq V$ must preserve units.

  Let us begin with a simple but powerful observation on the
  coprimality structure of the $k$-free Gaussian integers. Consider
  $x\in V\!$, with $\gcd^{}_{\mathrm{G}} (x,p)=1$ for every odd
  rational prime, where the $\gcd^{}_{\mathrm{G}}$ in $\ZZ[\ii]$ is
  unique up to units. Then, $p^{\ell} x \in V$ for any
  $1\leqslant \ell < k$, hence also
  $A(p^{k-1} x) = p^{k-1} A(x) \in V\!$, which implies
  $\gcd^{}_{\mathrm{G}} (A(x), p) = 1$.  This argument cannot be
  extended to $p=2=-\ii\ts (1+\ii)^2$, which is ramified.
  Nevertheless, we may conclude that
\[
   A (U) \, \subseteq \, U \cup (1+\ii)\ts U \cup \dots
   \cup (1+\ii)^{k-1} U \ts ,
\]
where we now need to exclude all but the first set on the
right-hand side.

Observe that, when $A$ is a mapping as specified, then so is the
mapping $A'$ defined by $A' (x) = \epsilon A(x)$, for any
$\epsilon \in U$.  We may thus assume $A(1) = (1+\ii)^m$ for some
$0\leqslant m \leqslant k-1$ without loss of generality, matched by
$A(\ii) = \kappa \ts (1+\ii)^n$ with $\kappa \in U$ and
$0\leqslant n \leqslant k-1$. Now, from $\ZZ$-linearity in conjunction
with bijectivity on $\ZZ[\ii]$, we know that $\det (A) =\pm 1$, where
\[
  \det (A) \, = \, \imag \bigl( \,
  \overline{\nts \nts A(1)} A(\ii) \bigr)
  \ts = \, \imag \bigl(\kappa \ts (1-\ii)^m (1+\ii)^n \bigr).
\]    
When $n\geqslant m$, this gives
$\det (A) = 2^m \imag \bigl( \kappa \ts (1+\ii)^{n-m} \bigr)$, which
cannot be unimodular unless $m=0$, so $A(1)=1$ and
$\det (A) = \imag \bigl( \kappa \ts (1+\ii)^n \bigr)$.

Observing $(1+\ii)^2 = 2 \ii$, an analogous argument now also excludes
$n\geqslant 2$, so $A(\ii)= \kappa$ or $A(\ii)=\kappa\ts (1+\ii)$. In
the first case, we get $A(\ii)=\ii$ or $A(\ii)=-\ii$ from bijectivity,
and $A$ is also a bijection on $U$. When $A(\ii)=\kappa\ts (1+\ii)$,
we get $A(1\pm \ii)= A(1) \pm A(\ii) = 1\pm \kappa\ts
(1+\ii)$. Irrespective of which unit $\kappa$ is, one of the images is
an element of norm $5$, where the norm refers to the \emph{field
norm}\footnote{Note that the absolute norm of an ideal in $\ZZ[\ii]$,
  which is always principal, agrees with the field norm of its
  generating element in this case.} of $x\in\QQ (\ii)$, which is
defined by $N(x) = x \bar{x}$ as usual.  But such a norm value is
impossible by our previous coprimality argument, and thus rules out
this case.

When $m > n$, a completely analogous chain of arguments gives $n=0$
and $m=1$, which is then once again ruled out by the coprimality
result. This leaves us with the mappings that preserve $U$ as claimed.
\end{proof}  

This result has the following immediate consequence, which can be seen
as a simplified (and purely algebraic) case of Fact~\ref{fact:Blomer}.

\begin{coro}\label{coro:bad-prime}
  Let\/ $k\geqslant 2$ be a fixed integer and\/
  $V\nts\nts =V^{(k)}_{\mathrm{G}}$ the set of\/ $k$-free Gaussian
  integers. If\/ $A\in\GL(2,\ZZ)\setminus \stab (V)$, there exists a
  Gaussian prime\/ $\rho$ and an element\/ $w\in V$ such that\/
  $\rho^k$ divides\/ $A (w)$.

  No such prime can be inert, and it cannot be ramified when\/ $k$ is
  even.
\end{coro}

\begin{proof}
  By Lemma~\ref{lem:gauss-group}, we know that $A\in \GL (2,\ZZ)$ with
  $A(V)\subseteq V$ must be an element of $\stab (V)$, from which the
  first statement is clear.

  The matrix $A$ is unimodular modulo $p^k$ for any rational prime
  $p$.  As such, it cannot change the number of cosets of $p^k \ZZ^2$,
  and maps the zero coset onto itself. This rules out the case that
  $\rho\in \cP^{}_{\mathrm{G}}$ is inert.

  If $\rho = 1 +\ii$, we have $N(\rho)=2$, and the same argument
  applies to $\rho$ when $k$ is even.
\end{proof}

Under the identification of $\ZZ[\ii]$ with $\ZZ^2$, let us now
consider the subshifts
\[
    \XX^{(k)}_{\mathrm{G}} \, \defeq \,
    \overline{\ZZ^2 + V^{(k)}_{\mathrm{G}}} \ts ,
\]    
which share many properties with our previous examples. In particular,
they once again satisfy $\XX^{(k)}_{\mathrm{G}} = \AAA$, with the
appropriate notion for admissibility, and are hereditary.  Further,
they have pure point spectrum with trivial topological point spectrum,
and the sets $V^{(k)}_{\mathrm{G}}$ are generic elements for the
corresponding patch frequency (or Mirsky) measure, the latter defined
via any averaging sequence of growing balls centred at $0$.

\begin{prop}\label{prop:G-one}
  Let\/ $\bigl( \XX^{(k)}_{\mathrm{G}} , \ZZ^2 \bigr)$ with fixed\/
  $k\geqslant 2$ be the faithful shift generated by the\/ $k$-free
  Gaussian integers.  Then, its centraliser is trivial, $\cS = \ZZ^2$,
  while the normaliser\/ $\cR$ consists of affine transformations
  only. In particular, $\cR$ contains a subgroup of the form\/
  $\ZZ^2 \rtimes D_4$, where\/
  $D_4 = \stab \bigl( V^{(k)}_{\mathrm{G}} \bigr)$ is the group from
  Lemma~\textnormal{\ref{lem:gauss-group}}.
\end{prop}

\begin{proof}
  The claim on the symmetries is a consequence of our general result
  in Theorem~\ref{thm:alg-symm}, which asserts that the centraliser is
  trivial, so $\cS = \ZZ^2$.
   
  For the extended symmetries, we are once more in the situation of
  the diagram \eqref{eq:cd-4} from the proof of
  Theorem~\ref{thm:alg-symm}.  Consequently, by
  Corollary~\ref{coro:affine}, each element of the normaliser is an
  affine mapping, namely an element of the affine lattice group
  $\ZZ^2 \rtimes \GL(2,\ZZ)$.
  
  That $\ZZ^2 \rtimes D_4$ is a subgroup of $\cR$ follows from
  Lemma~\textnormal{\ref{lem:gauss-group}}. Indeed, since the
  $\ZZ^2$-orbit of $V^{(k)}_{\mathrm{G}}$ is dense in
  $\XX^{(k)}_{\mathrm{G}}$ by construction and each element of $\cR$
  is continuous, any $M \in D_4$ maps $\XX^{(k)}_{\mathrm{G}}$ onto
  itself, as does any affine mapping $(t,M)$ with $t\in\ZZ^2$ and
  $M\in D_4$.
\end{proof}

It remains to complete the determination of $\cR$, which leads to the
following result.

\begin{theorem}\label{thm:Gauss}
  The symmetry group and the extended symmetry group of\/
  $\bigl(\XX^{(k)}_{\mathrm{G}},\ZZ^2 \bigr)$, with fixed\/
  $k\geqslant 2$, are given by\/ $\cS = \ZZ^2$ and\/
  $\cR = \cS \rtimes D_4$, respectively, where\/
  $D_4 = \stab (V) = C_4 \rtimes C_2$ is the symmetry group of the
  square, and as such a maximal finite subgroup of\/ $\GL (2,\ZZ)$. In
  particular, $C_4 \simeq \ZZ [\ii]^{\times}$, while\/ $C_2$ is the
  group of field automorphisms of\/ $\QQ(\ii)$, generated by complex
  conjugation.
\end{theorem}

\begin{proof}
  The role of $\ZZ^2 \rtimes D_4$ is clear from
  Proposition~\ref{prop:G-one}.  To complete the proof, we need to
  show that the only $\ZZ$-linear, bijective mappings of
  $\XX^{(k)}_{\mathrm{G}}$ onto itself are the ones we already know
  from Lemma~\ref{lem:gauss-group}.

  As in the case of $k$-free lattice points, now by
  Theorem~\ref{thm:alg-symm}, we have $\XX^{(k)}_{\mathrm{G}} = \AAA$,
  where $\AAA$ is the subshift that consists of all admissible subsets
  of $\ts V\!  =V^{(k)}_{\mathrm{G}}$. Here, $\nts V$ itself has the
  property that, for any $\pi\in\cP^{}_{\mathrm{G}}$, precisely the
  zero coset of the principal ideal $(\pi^k)$ is missing.
  
  To complete the proof, we have to show that no $\ZZ$-linear
  bijection of $\ZZ[\ii]\simeq \ZZ^2$ outside of $\stab (V)$ can map
  $\AAA$ into itself.  So, let $A\in\GL(2,\ZZ) \setminus \stab
  (V)$. Then, by Corollary~\ref{coro:bad-prime}, there is a
  $\rho\in\cP^{}_{\mathrm{G}}$ and an element $w\in V$ such that
  $\rho^k | A(w)$. Set $n=N(\rho)^k$ and $z^{}_{1} = w$. We will now
  choose Gaussian integers $z^{}_{2}, \ldots , z^{}_{n}$ such that the
  set $S=\{z^{}_{1}, z^{}_{2}, \ldots , z^{}_{n}\}$ is admissible
  while $A(S)$ meets all cosets of the principal ideal $(\rho^k)$ in
  $\ZZ[\ii]$.

  To this end, choose a non-empty, finite set $P$ of Gaussian primes
  that contains all primes with $N(\pi) < N(\rho)$ but none with
  $N(\pi) = N(\rho)$. Concretely, when $N(\rho)>2$, we just take all
  primes of smaller norm, while we simply choose the inert prime $3$
  when $\rho = 1+\ii$. In any case, we have
  $P = \{ \pi^{}_{1}, \ldots , \pi^{}_{m} \}$ with $m\geqslant 1$ this
  way.

  Let $\cL = (\pi^{k}_{1}\cdots \pi^{k}_{m} )$, which is a sublattice
  of $\ZZ[\ii]$ of index $N(\pi^{}_{1} \cdots \pi^{}_{m} )^k$. Since
  this index is coprime with $n=N(\rho)^k$, we know from
  Fact~\ref{fact:cosets} that $\cL$ meets all cosets of
  $A^{-1}(\rho^k)$, and so does $1+\cL$, as this is just a
  translate. So, select numbers
  $z^{}_{2}, \ldots , z^{}_{n} \in 1 + \cL$ such that
  $A(z^{}_{2}), \ldots, A(z^{}_{n})$ meet all non-zero cosets of
  $(\rho^k)$, and set $S\defeq\{ z^{}_{1}, \ldots , z^{}_{n} \}$, with
  $z^{}_{1} =w$.  Clearly, the set $A(S)$ now meets \emph{all} cosets
  of $(\rho^k)$ and is thus \emph{not} admissible for $\rho$, so
  $A(S) \not\in \AAA$.  If we can show that $S$ itself is admissible
  for all Gaussian primes, we are done.

  Clearly, $S$ is admissible for all Gaussian primes $\pi$ with
  $N(\pi) > N(\rho)$ by cardinality. If $S$ meets all cosets of
  $(\rho^k)$, each of them must occur precisely once. Then, we modify
  $S$ via replacing $z^{}_{2}$ by $z^{\prime}_{2} = z^{}_{2} + w$,
  which reduces the number of cosets in $S$ by one, without reducing
  the number of cosets in $A(S)$ because $w$ is $k$-free with
  $A(w) \equiv 0 \bmod (\rho^k)$.

  If $\rho$ is a splitting prime, we also have to check $\bar{\rho}$,
  which is not an associate but has the same norm.  If $S$ meets all
  cosets of $(\bar{\rho}^k)$, we need to modify one element $z_i$ with
  $i>1$ to remove one coset from $S$.  Due to the previous step, we
  can neither use $z^{\prime}_{2}$ nor the other element of $S$ that
  is now congruent to $z^{\prime}_{2}$ modulo $(\rho^k)$. Since
  $n\geqslant 4$, there is at least one other element, $z^{}_{4}$ say,
  that can be replaced by $z^{}_{4} + w$. The new set $S$ is now
  admissible for all Gaussian primes of norm at least $N(\rho)$, while
  $A(S)$ still meets all cosets of $(\rho^k)$ and is thus not in
  $\AAA$.

  If $\rho \ne 1 +\ii$, it remains to see whether $S$ is now also
  admissible for all $\pi$ with $N(\pi)<N(\rho)$. By our construction
  with the lattice $\cL$, we know that, modulo $(\pi^k)$, all $z_i$
  are congruent to $w$, $1$ or $1+w$, so we meet at most $3$
  cosets. Since $N(\pi)^k \geqslant 2^k \geqslant 4$, we are good, and
  $S$ is admissible for all Gaussian primes, while $A(S)$ is not, and
  we have the desired contradiction.
\end{proof}

A completely analogous chain of arguments works for the ring of
Eisenstein integers, $\ZZ[\rho]$, where
$\rho = \ee^{2 \pi \ii/3} = \frac{1}{2}(-1 + \ii \sqrt{3}\,)$ is a
primitive third root of unity.  This is the ring of integers in the
imaginary quadratic field $\QQ (\rho)$, and is again Euclidean. The
unit group is
\[
  \ZZ[\rho]^{\times} \, = \, \{ (-\rho)^m :
  0 \leqslant m \leqslant 5 \} \, \simeq \, C_6 \ts , 
\]
while the Eisenstein primes \cite[Thm.~255]{HW}, up to units,
are represented by
\[
  \cP^{}_{\mathrm{E}} \, = \, \{ 1-\rho \} \cup
  \{ p \in \cP : p \equiv 2 \bmod{3} \} \cup
  \{ \pi, \bar{\pi} : \pi \bar{\pi} = p \in \cP
  \text{ with } p \equiv 1 \bmod{3} \} \ts ,
\]
again in the order of the ramified prime, where
$(1-\rho)^2 = - 3 \ts \rho$, the inert primes, and the complex
splitting primes, where one pair $(\pi, \bar{\pi})$ is selected for
each $p$ in the last set.

Defining $V^{(k)}_{\mathrm{E}}$ for fixed $k\geqslant 2$ as the set of
$k$-free Eisenstein integers, which we may either view as a subset of
the triangular lattice, which is $\ZZ[\rho]$, or (equivalently) as one
of the square lattice via
$\{ (m,n) \in \ZZ^2 : m + n \rho \in V^{(k)}_{\mathrm{E}} \}$, the
analogue of Lemma~\ref{lem:gauss-group} now gives mappings of the form
$A (x) = \epsilon \ts \sigma(x)$ with
$\epsilon \in \ZZ[\rho]^{\times}$ and
$\sigma \in \{ \mathrm{id}, \bar{\cdot}\ts \}$, hence the group
$D_6 \simeq C_6 \rtimes C_2$, which is another maximal finite subgroup
of $\GL(2,\ZZ)$, this time the one that is the symmetry group of the
regular hexagon.

Defining the subshifts
\[
    \XX^{(k)}_{\mathrm{E}} \, \defeq \,
    \overline{\ZZ[\rho] + V^{(k)}_{\mathrm{E}}} \ts ,
\]    
one obtains the following analogue of Theorem~\ref{thm:Gauss}, the
proof of which need not be repeated, as the method is the same.

\begin{theorem}\label{thm:Eisenstein}
  The symmetry group and the extended symmetry group of\/
  $\bigl(\XX^{(k)}_{\mathrm{E}}, \ZZ [\rho]\bigr)$, with fixed\/
  $k\geqslant 2$, are given by\/ $\cS = \ZZ[\rho]\simeq\ZZ^2$ and\/
  $\cR = \cS \rtimes D_6$, respectively, where\/
  $D_6 = C_6 \rtimes C_2$ is the symmetry group of the regular
  hexagon, and as such isomorphic to a maximal finite subgroup of\/
  $\GL (2,\ZZ)$. In particular, $C_6 = \ZZ [\rho]^{\times}$, while\/
  $C_2$ is the group of field automorphisms of\/ $\QQ(\rho)$,
  generated by complex conjugation.  \qed
\end{theorem}

So far, we have seen extension groups that are either all of
$\GL (2,\ZZ)$ (for the visible lattice points), or finite subgroups
thereof (for the $k$-free Gaussian or Eisenstein integers).  In
particular, the subshifts defined by the two examples illustrated in
Figure~\ref{fig:vis+gau} are clearly distinguished by different
extended symmetry groups.  At this point, it is a natural question
whether also infinite true subgroups of $\GL (2,\ZZ)$ may occur. To
this end, we take a look at the corresponding dynamical systems for
\emph{real} quadratic fields.

\section{Power-free integers in real quadratic
  number fields}\label{sec:real}

Let us first consider subsets of $\ZZ^2$ constructed by means of
$k$-free integers in $\ZZ [\sqrt{2}\, ]$, namely
\[
  V^{(k)}_{2} \, \defeq \, \bigl\{ (m,n) \in \ZZ^2 : m + n \sqrt{2}
  \text{ is $k$-free in } \ZZ[\sqrt{2}\, ] \bigr\} ,
\]
where $k\in\NN$ with $k\geqslant 2$ is fixed.  This set emerges via
the isomorphism between $\ZZ^2$ and the Minkowski embedding of
$\ZZ[\sqrt{2}\, ]$ into $\RR^2$; compare \cite[Sec.~3.4.1]{TAO}. Here,
with $\lambda \defeq 1+\sqrt{2}$ denoting the fundamental unit, the
unit group is
\[
  U \nts \, = \: \ZZ [\sqrt{2} \, ]^{\times} = \,
  \{ \ts \pm \lambda^n : n \in \ZZ \}
  \, \simeq \, C_2 \times C_{\infty} \ts ,
\]
where we also note that $\ZZ[\sqrt{2}\,] = \ZZ[\lambda]$. This ring is
again Euclidean, so we can work with unique prime decomposition up to
units.

The primes \cite[Thm.~256]{HW} can be represented as
\[
  \cP_{2} \, = \, \{ \sqrt{2} \, \} \cup
  \{ p\in \cP : p \equiv \pm 3 \bmod{8} \} \cup
  \{ \pi , \pi^{\star} : \pi \ts \pi^{\star} = p \in \cP
  \text{ with } p \equiv \pm 1 \bmod{8} \} \ts ,
\]
where $(\cdot)^{\star}$ denotes the mapping that is the unique
extension of $\sqrt{2} \mapsto - \sqrt{2}$ to a field automorphism of
the quadratic field $\KK=\QQ (\sqrt{2}\, )$. The relevant field norm
is then given by $N (x) = x \ts x^{\star}$, which means
$N (m+n \sqrt{2}\, ) = m^2 - 2 \ts n^2$ or, equivalently,
$N ( r + s \lambda) = r^2 + 2 \ts r s - s^2$. Once again, to gain a
representation modulo units (integers of norm $\pm 1$), one pair is
selected in the last set for each $p$. Note that the field norm can be
negative here, wherefore the absolute norm of a principal ideal now is
the absolute value of the field norm of a generating element.

For some of the calculations below, it is helpful to express
$\lambda^n$ in terms of $\lambda$ and $1$, for arbitrary $n\in
\ZZ$. Defining the bi-infinite sequence $(c^{}_{n})^{}_{n\in \ZZ}$ by
the recursion $c^{}_{n+1} = 2 \ts c^{}_{n} + c^{}_{n-1}$ with initial
conditions $c^{}_{0} = 0$ and $c^{}_{1} = 1$, one obtains the analogue
of the Fibonacci numbers for the quadratic field $\KK$. In particular,
they satisfy $c^{}_{-n} = (-1)^{n+1} c^{}_{n}$ for all $n\in\ZZ$, and
the first few numbers are
\[
      \ldots, 29, -12, 5, -2, 1, 0, 1, 2, 5, 12, 29, \ldots
\]
The required formula for the units now reads
\begin{equation}\label{eq:fib}
     \lambda^n \, = \, c^{}_{n} \lambda + c^{}_{n-1}
     \, = \, c^{}_{n} \sqrt{2} + ( c^{}_{n} + c^{}_{n-1} ) \ts ,
\end{equation}
which holds for all $n\in\ZZ$, as can easily be checked by induction.

\begin{lemma}\label{lem:w-group}
  Let\/
  $A \! : \, \ZZ[\sqrt{2}\,] \xrightarrow{\quad} \ZZ[\sqrt{2}\,]$ be
  a\/ $\ZZ$-linear bijection that maps\/ $V\nts = V^{(k)}_{2}$ into
  itself, for some fixed integer\/ $k\geqslant 2$. Then, $A$ is of the
  form\/ $A (x) = \epsilon \ts\ts \sigma (x)$ with\/ $\epsilon \in U$
  and\/ $\sigma \in \{ \mathrm{id}, (\cdot)^{\star} \}$, so maps\/
  $\ts U \nts = \ZZ[\sqrt{2}\,]^{\times}$ onto itself.  Together,
  these mappings form the group\/
  $\stab (V) = U \nts \rtimes C_2 = C_2 \times ( C_{\infty} \rtimes
  C_2) = C_2 \times D_{\infty}$ of infinite order.
\end{lemma}

\begin{proof}
  Any $A$ of the form $A (x) = \epsilon \ts \sigma(x)$ satisfies
  $A(V)=V$ and maps $U$ onto itself, while the converse direction will
  be a consequence of showing that no further $\ZZ$-linear bijection
  of $\ZZ [\sqrt{2}\, ]$ exists that maps the set $V$ into itself.
  
  So, let $A$ be a $\ZZ$-linear bijection of $\ZZ [\sqrt{2}\,]$
  with $A(V)\subseteq V$.
  As in the proof of Lemma~\ref{lem:gauss-group}, we observe that
  $x\in V$ with $\gcd^{}_{\KK} (x,p)=1$ for any odd $p\in \cP$ implies
  $p^{k-1} x \in V$ and $A( p^{k-1} x) = p^{k-1} A(x) \in V$, hence
  $\gcd^{}_{\KK} ( A(x), p) = 1$ as well. Since $2 = (\sqrt{2}\, )^2$,
  which is the only ramified prime in this case, we see that $A(1)$
  and $A (\sqrt{2}\,)$ must be elements of the union
\[
    U \cup \sqrt{2} \ts U \cup 2 \ts U \cup \ldots
     \cup (\sqrt{2} \,)^{k-1} U \ts ,
\]  
where we may assume that we have, once again without loss of
generality, $A(1) = {2}^{m/2}$ and
$A(\sqrt{2}\,) = \kappa \ts\ts {2}^{n/2}$ with $\kappa \in U$ and
$0 \leqslant m,n \leqslant k-1$. Here, we also know that this must
result in a mapping with determinant $\pm 1$.
  
Now, define $W \! : \, \QQ(\sqrt{2}\, ) \xrightarrow{\quad} \QQ$ by
$W (x) = \bigl( x- x^{\star}\bigr)/2 \sqrt{2}$, and observe that this
gives $\det (A) = W \bigl( A(1)^{\star} \ts A(\sqrt{2}\,)\bigr)$,
hence
\[
  \det (A) \, = \, W \bigl( \kappa \ts
  (- 1)^m (\sqrt{2}\, )^{m+n} \bigr).
\]  
When $m+n $ is even, so $m+n = 2 \ell$, this means
$\det (A) = (-1)^m \ts 2^{\ell} \ts W (\kappa)$, which can only be
unimodular if $\ell = 0$ and thus $m=n=0$. With
$\kappa = \pm \lambda^r = \pm \bigl( c^{}_{r} \sqrt{2} + (c^{}_{r} +
c^{}_{r-1}) \bigr)$ from \eqref{eq:fib}, we then get
$\det (A) = \pm c^{}_{r}$, which in turn implies $c^{}_{r} = 1$ and
thus $r=\pm 1$. So, we have to consider $A(1) = 1$ together with
$A(\sqrt{2}\,) = \pm \lambda$. Both choices, however, lead to a
contradiction to our coprimality condition by observing that
$2 \pm \sqrt{2}$, which has norm $2$, is then mapped under $A$ to
$3+\sqrt{2}$, which is a number of norm $7$.
  
Likewise, when $m+n = 2 \ell +1$, we have
$\det (A) = (-1)^m \ts 2^{\ell} \ts W (\kappa \sqrt{2}\,)$, which
forces $\ell = 0$ and thus either $m=1$ and $n=0$ or $m=0$ and $n=1$.
In both cases, $\kappa = \pm \lambda^r$ can only lead to a unimodular
determinant when $c^{}_{r} + c^{}_{r-1} \in \{ \pm 1 \}$, which means
$r \in \{ -1, 0, 1 \}$. When $m=1$ and $n=0$, we get $A(1) = \sqrt{2}$
together with
$A(\sqrt{2}\,) = \kappa \in \{ \pm \lambda, \pm 1, \pm \lambda^{\star}
\}$. All six choices lead to contradictions to coprimality with odd
primes, by considering images of $1\pm \sqrt{2}$ or $2\pm \sqrt{2}$
under $A$.
  
It remains to consider $m=0$ and $n=1$, so $A(1)=1$ together with
$A(\sqrt{2}\,) = \kappa \sqrt{2}$, with the same options for $\kappa$
as in the previous case. Once again, $\kappa = \pm \lambda$ and
$\kappa = \pm \lambda^{\star}$ are impossible, as can be seen by
considering $A(1 \pm \sqrt{2}\,)$. The choices $\kappa = \pm 1$,
however, give the mappings $A(x) = x$ and $A(x) = x^{\star}$, which
map $U$ onto itself, as does any multiplication of such an $A$ with an
arbitrary $\epsilon \in U$.
\end{proof}

Let us now consider the subshifts
$ \XX^{(k)}_{2} \defeq \overline{\ZZ^2 + V^{(k)}_{2}} $, in complete
analogy to above.

\begin{prop}\label{prop:w2-sum}
  The symmetry group and the extended symmetry group of\/
  $\bigl( \XX^{(k)}_{2}, \ZZ^2 \bigr)$, with fixed\/ $k\geqslant 2$,
  are given by\/ $\cS = \ZZ^2$ and\/ $\cR = \cS \rtimes \cH$,
  respectively, where the extension group is\/
  $\cH = \stab \bigl( V^{(k)}_{2} \bigr) = U\nts \rtimes C_2 \simeq
  C_2 \times \nts D_{\infty}$, which is infinite.
\end{prop}

\begin{proof}
  From Lemma~\ref{lem:w-group}, we see that $\ZZ^2 \rtimes \cH$ is a
  subgroup of $\cR$. The latter is a subgroup of
  $\ZZ^2 \rtimes \GL(2,\ZZ)$ by Corollary~\ref{coro:affine}.  It thus
  remains to show that $\cH = \stab (V^{(k)}_{2})$ contains all
  $\GL (2,\ZZ)$ elements that map $\XX^{(k)}_{2}$ into itself.  This
  last step can be established by the method from the proof of
  Theorem~\ref{thm:Gauss}, with the field norm replaced by the
  absolute norm.
\end{proof}

There are other real quadratic fields that are Euclidean, such as
$\QQ (\sqrt{m}\, )$ with \mbox{$m=5$} and $m=3$, which play prominent
roles in the theory of aperiodic order, as they are connected with
systems with fivefold and twelvefold symmetry, respectively; see
\cite[Sec.~2.5.1]{TAO} for background.

For $m=5$, the ring of integers is $\ZZ[\tau]$, where
$\tau = \frac{1}{2} (1 + \sqrt{5}\, )$ is the golden ratio. Its
unit group is
$U \nts = \ZZ[\tau]^{\times} = \{ \pm \ts \tau^n : n\in \ZZ \}$, and
the primes \cite[Thm.~257]{HW} are represented by
\[
  \cP^{}_{5} \, = \, \{ \sqrt{5} \, \} \cup
  \{ p\in\cP : p\equiv \pm 2 \bmod{5} \} \cup
  \{ \pi, \pi^{\star} : \pi \ts \pi^{\star} = p \in \cP
  \text{ with } p  \equiv \pm 1 \bmod{5} \} \ts ,
\]
where $(\cdot)^{\star}$ is the field automorphism of
$\QQ (\sqrt{5}\,)$ induced by $\sqrt{5} \mapsto - \sqrt{5}$, with our
usual convention for the splitting primes in place. The only ramified
prime is $5$, while the field norm on $\ZZ[\tau]$ is
$N(m+n\tau)=m^2 + mn - n^2$, which can be negative. \smallskip

Finally, let us consider the slightly more complicated case $m=3$,
where the ring of integers is $\ZZ[\sqrt{3}\,]$. Its unit group is
given by $\ZZ[\sqrt{3}\,]^{\times} = \{ \pm \ts \eta^n : n\in \ZZ \}$,
with fundamental unit $\eta = 2 + \sqrt{3}$. Here, in contrast to the
two previous cases, all units have norm $1$. Employing
\mbox{\cite[Thm.~11.1]{Zagier}}, one sees that the primes up to units
can be represented as
\[
\begin{split}
  \cP^{}_{3} \, = \, & \{ 1 +\sqrt{3}, \sqrt{3} \, \} \cup \{ p \in
  \cP : p \equiv \pm \ts 5 \bmod{12} \} \\ & \cup \{ \pi, \pi^{\star}
  : \pi \ts \pi^{\star} = \pm \ts p \in \cP \text{ with } p \equiv \pm
  1 \bmod{12} \} \ts
\end{split}  
\]
with the usual convention for the last set, where $(\cdot)^{\star}$ is
now induced by $\sqrt{3} \mapsto - \sqrt{3}$.  Unlike before, since
the field discriminant is $12$ and thus divisible by $2$ and $3$,
there are \emph{two} ramified primes, where
$(1+\sqrt{3}\,)^2 = 2 \ts \eta$ is the additional relation.

This leads to more cases to consider in the determination of
$\stab (V^{(k)}_{3})$, but the $\ZZ$-linear bijections of
$\ZZ[\sqrt{3}\,]$ that map $V^{(k)}_{3}$ into itself, for some fixed
$k\geqslant 2$, are still the expected ones, namely the maps $A$ of
the form $A (x) = \epsilon \ts \sigma (x)$ with
$\epsilon \in U\nts =\ZZ[\sqrt{3}\,]^{\times}$ and
$\sigma \in \{ \mathrm{id}, (\cdot)^{\star} \ts \}$; we leave this
proof to the interested reader.

In both cases, a proof analogous to the one of 
Proposition~\ref{prop:w2-sum} gives the following result.

\begin{theorem}\label{thm:real}
  The symmetry group and the extended symmetry group of\/
  $\bigl( \XX^{(k)}_{m} , \ZZ^2 \bigr)$, with fixed\/
  $m\in \{ 2, 3, 5\}$ and\/ $k\geqslant 2$, are given by\/
  $\cS = \ZZ^2$ and\/ $\cR = \cS \rtimes \cH$, respectively, where the
  extension group is\/
  $\cH = \stab \bigl( V^{(k)}_{m} \bigr) = U\nts\nts \rtimes C_2 \simeq
  C_2 \times D_{\infty}$, which is an infinite group that does not
  depend on\/ $k$, where\/ $U$ is the unit group as before.
  \qed
\end{theorem}

The advantage of using the normaliser in addition to the centraliser
as a topological invariant becomes obvious in dimensions
$d\geqslant 2$. In \cite{BRY}, this was demonstrated for the chair
tiling shift and for Ledrappier's shift. In both cases, $\cR$ was an
extension of $\cS$ of finite index. As our number-theoretic examples
above show, this phenomenon occurs again, but $\cR$ can also be an
\emph{infinite-index} extension of $\cS$, either for trivial reasons
(visible lattice points) or for non-trivial ones ($k$-free
$\ZZ [\sqrt{2}\,]$-integers).  At present, we do not know whether such
an infinite extension is also possible for minimal, deterministic
(zero entropy) subshifts. In any case, these groups allow the
distinction of several subshifts (up to topological conjugacy) that
have the same centraliser, but different normalisers, such as
$(\XX^{}_{V}, \ZZ^2)$ and $(\XX^{(2)}_{\mathrm{G}},\ZZ^2)$ from above.

Due to the nature of the associated dynamical system, the structure of
the hull $\XX^{}_V$(given the property of hereditariness and the
natural topology employed) allows for \emph{local symmetries} (that
is, transformations that preserve finite local substructures in the
set $V$ up to translation) to manifest themselves. To some extent,
they may be observed by analysing the extended symmetry group
$\cR(\XX^{}_V)$. While symmetries of the set $V$ in the standard sense
(global symmetries) are obviously local symmetries in this new sense,
the converse is not clear. It is easy to build sets $V$ that have many
local symmetries while lacking global symmetries entirely. Thus, it is
interesting to note that, in the current context, those two kind of
symmetries happen to be the same, bringing up the question on whether
this is a natural phenomenon on sets defined in an `algebraic' form in
more general ways.

This setting deserves further attention, in particular in the context
of dynamical systems of number-theoretic origin. As this will require
a more general approach via ideals, as well as some additional and
less elementary results from algebraic and analytic number theory, we
defer this to a separate investigation \cite{BBN}.

\section*{Acknowledgements}

MB would like to thank Valentin Blomer for helpful discussions and
hints on the manuscript.  We thank the referee for a number of
suggestions, which helped to improve the presentation.  Moreover, AB
is grateful to ANID (formerly CONICYT) for the financial support
received under the Doctoral Fellowship ANID-PFCHA/Doctorado
Nacional/2017-21171061, ML was supported by Narodowe Centrum Nauki
grant UMO-2019/33/B/ST1/00364, and AN acknowledges financial support
by the German Research Foundation (DFG) through its Heisenberg program
(project number 437{\ts}113{\ts}953).

\end{document}